\documentclass{amsart}
\usepackage{amsfonts,amsthm,amsmath}
\usepackage{amssymb}
\usepackage[utf8]{inputenc}
\usepackage[a4paper]{geometry}
\usepackage{epic}
\usepackage{enumerate}
\usepackage[usenames,dvipsnames]{color}
\numberwithin{equation}{section}
\numberwithin{equation}{subsection}
\theoremstyle{plain}
	\newtheorem{thm}[equation]{Theorem}
	
	\newtheorem{lemma}[equation]{Lemma}
	\newtheorem{prop}[equation]{Proposition}
	\newtheorem{cor}[equation]{Corollary}

	\newtheorem{remark}[equation]{Remark}
	\newtheorem*{remark*}{Remark}
	\newtheorem{example}[equation]{Example}
	\newtheorem{question}{Question}
\theoremstyle{definition}
	
\author{Tam\'as L\'aszl\'o}
\address{BCAM - Basque Center for Applied Math.,
Mazarredo, 14 E48009 Bilbao, Basque Country – Spain}
\email{tlaszlo@bcamath.org}

\author{Andr\'as N\'emethi}
\address{Alfr\'ed R\'enyi Institute of Mathematics,
Hungarian Academy of Sciences,
Re\'altanoda utca 13-15, H-1053, Budapest, Hungary \newline
 \hspace*{4mm} ELTE - University of Budapest, Dept. of Geometry, Budapest, Hungary \newline \hspace*{4mm}
BCAM - Basque Center for Applied Math.,
Mazarredo, 14 E48009 Bilbao, Basque Country – Spain}
\email{nemethi.andras@renyi.mta.hu }

\title{On the geometry  of strongly flat semigroups and their generalizations}

\begin{document}

\begin{abstract}
Our goal is to convince the readers that the theory of complex normal surface singularities 
can be  a powerful tool in  the study of numerical semigroups, and, in the same time,
 a very rich source of
interesting affine and numerical semigroups. More precisely, 
we prove that the strongly flat semigroups,  which satisfy the  maximality property with respect to the Diophantine Frobenius problem,  are exactly the numerical semigroups associated with negative definite Seifert homology spheres via the possible `weights' of the generic $S^1$--orbit. 
Furthermore, we consider their generalization to the Seifert rational homology sphere case and prove an explicit (up to a Laufer computation sequence) formula for their Frobenius number.
The singularities behind are the weighted homogeneous ones, whose several topological and analytical 
 properties are exploited. 
\end{abstract}

% \setcounter{tocdepth}{1}
% \tableofcontents

\maketitle

\begin{center}
{\em  Dedicated to L\^e D\~ung Tr\'ang} 
\end{center}

\section{Introduction}

%\subsection{} 
This note creates a bridge between the theory of numerical semigroups and 
the theory of complex normal surface singularities. By several examples we suggest how one of the two theories might provide meaningful and  enriching questions, ideas and simultaneously 
powerful tools to the  other.  

For example, one of the most classical and important problem in the theory of numerical semigroups is
the determination of the (minimal set of) generators and also their Frobenius number. 
These are very difficult  algebraic/combinatorial problems.
In this note we provide a new tool,  based on techniques of singularity theory, which solves these
problems for several semigroups (which can be related with singularities).

In theory of surface singularities two semigroups appear very naturally. If $(X,o)$
is a normal surface singularity, and we fix a good resolution of it with 
$n$ irreducible exceptional curve, then one defines the `Lipman cone' $\mathcal{S}_{top}$, 
an affine monoid, submonoid of $\mathbb{N}^n$. It can be determined completely from the 
combinatorics of the resolution graph, in particular, it is a topological invariant associated with a 
fixed plumbing graph of the oriented (3--manifold) link of the    singularity. 

The second monoid, $\mathcal{S}_{an}$, still associated with the fixed resolution,
 is a submonoid of $\mathcal{S}_{top}$, and it is determined by the
analytic structure of the singularity. In general is very hard to determine it. 

The projection of $\mathcal{S}_{an}$
to any of the  coordinates of $\mathbb{N}^n$ (given by the choice of one of the
irreducible exceptional curves) provides an interesting  numerical semigroup. It is a real challenge 
to compute its generators,  Frobenius  number, list its properties (eg. it is symmetric 
or not). 

In this note we take the case of weighted homogeneous surface singularities.
 Their links are oriented Seifert 3--manifolds. In particular, the minimal good 
  resolution graph (or, the minimal plumbing graph of the link) 
  is star--shaped. We will assume that the link $M$ is 
 rational homology sphere, that is, all the genus--decorations of the graph are zero. 
 We also take the generic case, when the number of `legs' of the star--shaped graph is 
 $\geq 3$. Furthermore, we will choose  for the projection the central vertex (which, in Seifert
 geometry,  corresponds  to the generic $S^1$--orbit). 
We define $\mathcal{S}_M $ as the projection of $\mathcal{S}_{an}$. 

Our starting point in the construction of the bridge mentioned above is to show that 
$\mathcal{S}_M$ appears naturally in the classical theory of numerical semigroups.
Indeed, Raczunas and Chrz\c astowski-Wachtel in \cite{RChW} characterized certain semigroups, which 
 realizes a sharp upper bound estimate for the Diophantine Frobenius problem/number. They called 
 them `strongly flat semigroups'. In this note we prove 
  that these semigroups are exactly the semigroups of type $\mathcal{S}_M $ associated with 
  Seifert integral homology spheres. 
  
  Then, we continue with the study of $\mathcal{S}_M $ in the rational homology sphere case. 
One of the final goals is the computation the Frobenius number. In order to determine this 
we define a $\mathcal{S}_M $--module as well (which probably was not considered in 
semigroup theory). The main tools are the topological properties of Seifert 3--manifolds,
Pinkham's theory of weighted homogeneous singularities, Laufer's computations sequences,
and results (valid for splice quotient singularities) regarding the generators of
$\mathcal{S}_{an}$. 

In the case of links of Briekorn--Hamm complete intersection we determine the minimal set of 
generators of $\mathcal{S}_M$ as well.

\subsection*{Acknowledgements}
TL was supported by ERCEA Consolidator Grant 615655 – NMST and
by the Basque Government through the BERC 2014-2017 program and by Spanish
Ministry of Economy and Competitiveness MINECO: BCAM Severo Ochoa
excellence accreditation SEV-2013-0323.
He is very grateful to the members of BCAM for the warm hospitality and for providing with him an excellent research environment during his postdoctoral stay in Bilbao.

AN  was partially supported by  NKFIH Grant  K112735.

\section{The Diophantine Frobenius problem and strongly flat semigroups}\label{s:DFP}

\subsection{}\label{ss:DFP}
The famous Diophantine Frobenius problem asks to find an explicit formula for the greatest integer not representable as a nonnegative linear form of a given system of $d$ relatively prime integers $1\leq a_1\leq \ldots \leq a_d$. The  integer defined in this way
is called the {\it Frobenius number}
 of the system, or of the numerical semigroup $G(a_1,\ldots,a_d)$,
  generated by the integers from the system itself. It will be denoted by $f_{G(a_1,\ldots,a_d)}$.

The very first result related to this problem is the well-known formula of Sylvester, namely  $f_{G(a_1,a_2)}=a_1a_2-a_1-a_2$ \cite{syl}.
Although  several formulas for peculiar systems and general bounds exist in the literature,
the problem is still open in full generality. In this note we will
recall/discuss almost nothing from the `classical combinatorial approach' ---
the interested reader might consult for more details
 the excellent monograph of Ram\'irez Alfons\'in \cite{RamAlf} ---, our goal
 is to connect the problem with singularity theory, and
to show the strength  of this new method.
  From this point of view the following sharp estimate  is a good
 starting  point:  Raczunas and Chrz\c astowski-Wachtel \cite{RChW} found the following
upper bound using  the least common multiple ${\rm lcm}(a_1,\ldots,a_d)$
\begin{equation}\label{fbound}
f_{G(a_1,\ldots,a_d)} \leq (d-1)\cdot {\rm lcm}(a_1,\ldots,a_d)-\sum_{i=1}^d a_i.
\end{equation}
Moreover, they have characterized the class of semigroups (or system of generators) for which the equality in (\ref{fbound}) holds. These are the so called strongly flat semigroups, which satisfy the property that for every $i$ one can write $a_i=\alpha_1\ldots\alpha_{i-1}\alpha_{i+1}\ldots\alpha_d$ with each $\alpha_i$ being the greatest common divisor of $a_1,\ldots,a_{i-1},a_{i+1},\ldots,a_d$.
(From this also follows that  ${\rm gcd}(\alpha_i,\alpha_j)={\rm gcd}(a_1,\ldots, a_d)=1 $ 
for $i\not=j$.)

In the sequel we will discuss the geometric and topological aspects of strongly flat semigroups and
also their natural topological generalizations. During this discussion we try to
 embed the solution of the Frobenius problem  into a more complex package
 of topological/geometrical invariants --- even to relate with more semigroups with richer structure.

\section{Normal surface singularities}\label{s:NSS}

\subsection{Notations and preliminaries}\label{ss:nsstop}
We consider a complex normal surface singularity $(X,o)$ whose link $M$ is a rational homology sphere. Fix a good resolution $\pi:\widetilde{X}\to X$ with dual graph $\Gamma$ whose vertices are denoted by $\mathcal{V}$. We use notation $E=\pi^{-1}(o)$ for the exceptional divisor and let $\{E_v\}_{v\in\mathcal{V}}$ be its irreducible components. Notice that $M$ is a rational homology sphere if and only if $\Gamma$ is a tree and all the $E_v$ are rational.

Then  $L:=H_2(\widetilde{X},\mathbb{Z})$ is a lattice freely generated by the classes of the irreducible exceptional divisors $\{E_v\}_{v\in\mathcal{V}}$, together with the nondegenerate negative definite intersection form $I=(E_v,E_w)_{v,w}$. They are exactly the integral cycles supported on $E$.
The determinant $\det(\Gamma)$ of the graph by convention is $\det(-I)$ (which is always positive).

If  $L'$ denotes
$H^2( \widetilde{X}, \mathbb{Z})$, then the intersection form provides an embedding $L \hookrightarrow
L'$ with factor $L'/L\simeq H_1(M,\mathbb{Z})$, denoted by $H$.
In the sequel $[l']$ denotes the class of $l'$. In fact,
$L'\simeq{\rm Hom}_{\mathbb{Z}}(L,\mathbb{Z})$, the dual lattice.
The form extends to $L\otimes \mathbb{Q}$ hence to
$L'$ too (via the natural inclusion
$L'\simeq\{l'\in L\otimes \mathbb{Q}\,:\, (l',L)\in \mathbb{Z}\}\subset  L\otimes \mathbb{Q}$).
The module $L'$ over $\mathbb{Z}$ is freely generated by the (anti-)duals $\{E_v^*\}_v$, where we
prefer the convention $( E_v^*,E_w) =  -1 $ for $v = w$, and
$0$ otherwise.

For $l'_1,l'_2\in L\otimes \mathbb{Q}$ with $l'_i=\sum_v l'_{iv}E_v$ for $i=\{1,2\}$ one writes $l'_1\geq l'_2$ if
$l'_{1v}\geq l'_{2v}$ for any $v\in\mathcal{V}$. In particular, $l'$ is effective if $l'\geq 0$. We set also
$\min\{l'_1,l'_2\}:= \sum_v\min\{l'_{1v},l'_{2v}\}E_v$.  Furthermore,
if $l'=\sum_v l'_v E_v$ then we write $|l'|:=\{v\in \mathcal{V}\,:\, l'_v\not=0\}$ for the
{\em support} of $l'$.

In several computations we need  the $E_v$--coefficient of  $E_u^*$; this equals
$-(E^*_u,E^*_v)=(-I^{-1})_{uv}$, which multiplied by $\det(\Gamma) $ is the determinant
of the subgraph obtained from $\Gamma$ by deleting the shortest path connecting $u$ and $v$ and the adjacent edges.

For more details on normal surface singularities and resolution graphs see eg. \cite{Nfive,NOSZ}.

\subsubsection{ Characteristic cycles}
The {\it canonical cycle} $K\in L'$ is defined by the
{\it adjunction formulae}
\begin{equation}\label{eq:adjun}
(K+E_v,E_v)+2=0 \ \ \ \mbox{ for all $v\in\mathcal{V}$.}
\end{equation}
Let $Char:=\{k\in L' \ : \ (k+l,l)\in 2\mathbb{Z} \ \ \mbox{for any} \ l\in L\}$ be the set of characteristic cycles in $L'$. By the adjunction formulae it can be written as $Char=K+2L'$. There is a natural action of $L$ on $Char$ given by $l\ast k:=k+2l$ whose orbits are of type $[k]:=k+2L$. Moreover, $H$ acts freely and transitively on the set of orbits by $[l']\ast [k]=[k+2l']$. Therefore, every orbit can be written in the form $[k]=K+2(l'+L)$ for some $l'\in L'$ with a fixed group element $[l']=h$ which indexes the orbit.
One has an identification between the orbits of $Char$ and the $spin^c$-structures of the link $M$,
see eg. \cite{NOSZ}.

Associated with any characteristic cycle $k$ we define the Riemann-Roch function
\begin{equation}\label{eq:RR}
\chi_k(l'):=-(k+l',l')/2 \ \ \ \mbox{for any } \ l'\in L'.
\end{equation}
For simplicity we will use the  notation $\chi:=\chi_K$. Notice that $\chi_k(l)$ ($l\in L$) by index
  (Riemann-Roch type) formula equals the  Euler characteristic of certain line bundles,
  see e.g. \cite[2.2.8]{Ngr}.

A topological type of singularity is called {\it numerically Gorenstein }
if $K\in L$ (this property does not depend on the choice of the resolution graph).
An analytic type of singularity is called {\it Gorenstein} if $K\in L$ and the sheaf of
holomorphic 2--forms $\Omega_{\widetilde{X}}$ is isomorphic with
$ {\mathcal O}  _{\widetilde{X}}(K)$.

\subsubsection{Minimal and distinguished cycles}\label{ss:mincyc}
We define the Lipman cone by $\mathcal{S}_{top}':=\{l'\in L' \ | \ (l',E_v)\leq 0 \ \mbox{for all} \ v\in \mathcal{V}\}$,
the semigroup (monoid) of  anti--nef rational cycles from $L'$.
 It is generated over $\mathbb{Z}_{\geq 0}$ by the cycles $E^*_v$.
Using the fact that the intersection form is negative definite (or the determinental characterizations
from \ref{ss:nsstop}) one shows that
all the entries of $E^*_v$ are strict positive. Hence $\mathcal{S}'_{top}$ sits in the first quadrant.
Define also  $\mathcal{S}_{top}:=L\cap \mathcal{S}'_{top}$,  
the semigroup (monoid) of anti-nef integral cycles,
the integral Lipman cone. By its very definition it is an affine semigroup.

  For any $h\in H$ there is a
 unique minimal element of $\{l'\in L' \ | \ [l']=h\}\cap \mathcal{S}'_{top}$ 
 (guaranteed eg.  by Lemma \ref{lem:cs2}), 
 which will be denoted by $s_h$. 
 Furthermore, we can also consider the semi-open cube $\{\sum_v l'_v E_v\in L' \ | \ 0\leq l'_v<1\}$
 which contains a unique representative $r_h$ for every $h$ so that $[r_h]=h$.
One has  $s_h\geq r_h$, however,  in general,
 $s_h \neq r_h$ (this can happen even for star--shaped resolution graphs,
 see eg. \cite[Ex. 4.5.4]{Ngr}).

The minimal cycle $s_h$ defines a distuinguished characteristic cycle $k_r:=K+2s_h$ of the corresponding orbit $[K+2s_h]$ indexed by $h\in H$.

Sometimes it is more convenient to use $Z_K:=-K\in L'$ instead of $K$, eg. if the resolution graph
is minimal then $Z_K\in \mathcal{S}'_{top}$ (use the adjunction formulae (\ref{eq:adjun})).

\subsubsection{Principal and anti-nef cycles}\label{ss:princc}
Any analytic function $f:(X,o)\to (\mathbb{C},0)$ determines an effective principal divisor $(\pi\circ f)=(f)_{\Gamma}+St(f)$, where $(f)_{\Gamma}$ is a cycle in $L_{>0}$ and $St(f)$ is supported by the strict transform of $\{f=0\}$. We define the set of principal cycles by
\begin{equation}\label{eq:san}
\mathcal{S}_{an}:=\{(f)_{\Gamma} \, : \, f\in \mathfrak{m}_{(X,o)}\}.\end{equation}
Then, in fact, $(f_{\Gamma}, E_v)\leq 0$ for any $v\in \mathcal{V}$, hence $\mathcal{S}_{an}$ is a subsemigroup of $\mathcal{S}_{top}$.
 If $(X, o)$ is rational (and $\pi$ is arbitrary), or minimally elliptic (and
$\pi$ is minimal and good) then  $\mathcal{S}_{an}=\mathcal{S}_{top}$, but in general
$\mathcal{S}_{an}\neq \mathcal{S}_{top}$, see eg. \cite{NNP}.

In general, for an arbitrary analytic singularity type,
it is very difficult problem to decide if an element of $\mathcal{S}_{top}$
is principal or not
(this is not a topological/combinatorial problem, the answer definitely might depend on the
choice of the analytic structure supported by the fixed resolution graph).

\subsection{Seifert 3--manifolds as singularity links}\label{ss:seifert}

\subsubsection{} The negative definite intersection form $I$ together with the collection of genera
 $\{g(E_v)\}_v$ can be coded in a connected  dual resolution graph $\Gamma$ as well.
 The link of the singularity can be recovered from the graph by plumbing construction. Furthermore,
 by a theorem of Grauert
 \cite{GRa}, any such connected negative definite graph can be realized as dual resolution graph
 of some analytic singularity. In this note we focus on star--shaped graphs, their associated
 plumbed 3--manifolds are Seifert 3--manifolds, and they can be analytically realized by weighted homogeneous singularities 
(however,   they  might be realized by many other analytic structures as well).
First we provide some details and notation regarding the combinatorics of the graph.

\subsubsection{Seifert 3--manifolds.}\label{sss:seifert}
Assume that a resolution  graph $\Gamma$ is star--shaped with $d$ legs, $d\geq 3$. Each leg is a chain determined by the normalized Seifert invariant $(\alpha_i,\omega_i)$,
where $0<\omega_i <\alpha_i$, gcd$(\alpha_i,\omega_i)=1$ as follows. If we consider the Hirzebruch/negative continued fraction expansion
$$ \alpha_i/\omega_i=[b_{i1},\ldots, b_{i\nu_i}]=
b_{i1}-1/(b_{i2}-1/(\cdots -1/b_{i\nu_i})\cdots )\ \  \ \ (b_{ij}\geq 2),$$
then the $i^{\mathrm{th}}$ leg has $\nu_i$ vertices, say $v_{i1},\ldots, v_{i\nu_i}$,
with Euler decorations (self--intersection numbers)
$-b_{i1},\ldots, -b_{i\nu_i}$, where $v_{i1}$ is connected by the central vertex denoted by $v_0$.
All these vertices have genus--decorations zero.
We also use $\omega_i'$ satisfying $\omega_i\omega_i'\equiv 1$ (mod $\alpha_i$), $0< \omega_i'<\alpha_i$. (One shows that  $\alpha_i$ is the determinant of the $i^{th}$--leg $\Gamma_i$,
$\omega_i=\det(\Gamma_i\setminus v_{i1})$, and  $\omega_i'=\det(\Gamma_i\setminus v_{i\nu_i})$.)

The central vertex $v_0$ has an Euler decoration
$-b_0$ and genus decoration $g$.
 %and it corresponds to the central genus $g$ curve $E_0$ with self-intersection number $-b_0$.

The plumbed 3--manifold $M$ associated with such a star--shaped graph $\Gamma$ has a Seifert structure and its associated normalized Seifert invariants is denoted by $$Sf=(-b_0,g;(\alpha_i,\omega_i)_{i=1}^d).$$ In the sequel we will assume that $M$ is a rational homology sphere, or, equivalently, the central node has $g=0$.

The orbifold Euler number of $M$ is defined as $e=-b_0+\sum_i\omega_i/\alpha_i$. The negative definiteness of the intersection form is equivalent with  $e<0$. We also write  $\alpha:=\mathrm{lcm}(\alpha_1,\ldots,\alpha_d)$.

Let $\mathfrak{h}:=|H|$ be the order of $H=H_1(M,\mathbb{Z})=L'/L$, and
let  $\mathfrak{o}$ be the order of the class of $E^*_0$ (or of
the generic $S^1$ Seifert--orbit)   in $H$
(in the plumbing construction, this orbit is the $S^1$--fiber over a generic point of $E_0$).
 Then one has (see eg. \cite{neumann.abel})
\begin{equation}\label{eq:sei2}
 \mathfrak{h}=\alpha_1\cdots\alpha_d|e|, \ \ \ \ \mathfrak{o}=\alpha|e|.
\end{equation}
If $M$ is an integral homology sphere (that is, Seifert homology sphere) then necessarily
all $\alpha_i$'s are pairwise relative prime and (cf. (\ref{eq:sei2})) $\alpha |e|=1$.
This reads as the Diophantine equation
$(b_0-\sum_i\omega_i/\alpha_i)\alpha=1$, hence all $\omega_i$'s and $b_0$
are determined uniquely by the $\alpha_i$'s.
The corresponding Seifert 3--manifold (link) will be denoted by $\Sigma(\alpha_1,\dots,\alpha_d)$.

\subsubsection{Weighted homogeneous surface singularities}\label{ss:WHS}
Let $(X,o)$ be a normal weighted homogeneous surface singularity. It is the germ at the origin of an affine variety $X$ with good $\mathbb{C}^*$-action, which means that its affine coordinate ring is $\mathbb{Z}_{\geq 0}$--graded: $R_X=\oplus_{\ell\geq 0} R_{X,\ell}$.
The dual graph of the  minimal good resolution is star--shaped, and the $\mathbb{C}^*$-action
of the singularity induces an $S^1$--Seifert action on the link. In particular,
the link $M$ of $(X,o)$ is a negative definite Seifert 3--manifold characterized by its
normalized Seifert invariants $Sf=(-b_0,g;(\alpha_i,\omega_i)_{i=1}^d)$.
It is known that weighted homogeneous surface singularity with numerically Gorenstein
topological type is Gorenstein (in particular, eg., combinatorial consequences of the Gorenstein
duality can be read as combinatorial symmetries of the numerically Gorenstein graph).

The formulae below are attributed to Dolgachev, Pinkham and Demazure, hence
by some authors are called Dolgachev--Pinkham--Demazure formulae; we will follow here Pinkham's
work \cite{Pinkham}.
By  \cite{Pinkham} the complex structure is completely recovered by the Seifert invariants and the configuration of points $\{P_i:=E_0\cap E_{i1}\}_{i=1}^d \subset E_0$.
(Here $E_0$ is the irreducible exceptional curve indexed by the central vertex $v_0$.)
In fact, the graded ring of he local
algebra of the singularity is given by the formula
$R_X=\oplus_{\ell\geq 0} R_{X,\ell}=\oplus_{\ell\geq 0}
H^0(E_0,\mathcal{O}_{E_0}(D^{(\ell)}))$ with $D^{(\ell)}:=\ell (-E_0|_{E_0})-\sum_{i=1}^d\lceil\ell\omega_i/\alpha_i\rceil P_i$, where for $r\in\mathbb{R}$ $\lceil r\rceil$ denotes the smallest integer greater than or equal to $r$.

In particular, when $M$ is a rational homology sphere, ie. $g=0$, one has $E_0\simeq \mathbb{P}^1$, hence Pinkham's formula implies that $\dim(R_{X,\ell})=\max\{0,1+N(\ell)\}$ is topological, where $N(\ell)$
is  the quasi-linear function
\begin{equation}\label{defN}
N(\ell):=\deg D^{(\ell)}=b_0\ell-\sum_{i=1}^d\Big\lceil \frac{\ell \omega_i}{\alpha_i}\Big\rceil.
\end{equation}
Since $-\lceil x\rceil \leq -x$ one obtains $N(\ell)\leq |e|\ell$,
hence  $N(\ell)<0$ for $\ell<0$.
Therefore, if we consider the set $\mathcal{S}_M:=\{\ell\in\mathbb{Z}_{\geq 0} | R_{X,\ell}\neq 0\}$, which is a numerical semigroup by the grading property of the algebra, by the above results it can be characterized topologically with the Seifert invariants  as follows
\begin{equation}\label{eq:sgptop}
\mathcal{S}_M=\{\ell\in\mathbb{Z} \ | \ N(\ell)\geq 0\}.
\end{equation}
It is called the {\it semigroup of  $(X,o)$,  or,
 of the Seifert rational homology sphere link $M$}.
In particular,  the Poincar\'e series $
P_0(t)=\sum \dim_{\mathbb{C}}(R_{X,\ell}) t^{\ell}$
of the graded local algebra is also topological, namely
\begin{equation}\label{eq:poincare}
P_0(t)=\sum_{\ell\geq 0}\max\{0,1+N(\ell)\}t^{\ell}.\end{equation}
Its support is exactly $\mathcal{S}_M$. For more (from both analytic and topological
points of view), see \cite{NOSZ,Nsplice} and \cite{LN}.
By a similar pattern, one can define the following series as well:
\begin{equation}\label{eq:polp}
P_0^+(t)=\sum_{\ell\geq 0} \max\{0, -1-N(\ell)\} t^{\ell}, \ \ \mbox{and}\ \
P_0^{neg}(t)=\sum_{\ell\geq 0}(1+N(\ell))t^{\ell}.
\end{equation}
$P_0^+$ has two interpretations. Firstly, since $E_0\simeq \mathbb{P}^1$, one has
$\max\{0, -1-N(\ell)\}=\dim\, H^1(E_0,\mathcal{O}_{E_0}(D^{(\ell)}))$, hence
$P^+_0(t)=\sum_{\ell\geq 0} h^1(E_0,\mathcal{O}_{E_0}(D^{(\ell)}))t^{\ell}$.
Since $N(\ell)$ asymptotically behaves like $|e|\ell$ (see also Prop \ref{prop:INEQ}
below), $P_0^+(t)$ is a polynomial.
By \cite{Pinkham}, $P^+_0(1)$ equals the geometric genus $p_g$ of the weighted homogeneous singularity.
From topological point of view,
 $P^+_0(1)$ equals the ($h=0$)--equivariant part of the normalized Seiberg--Witten invariant of the link (cf. \cite{NOSZ}).

It is clear that both $P_0$ and $P_0^+$ are determined from  $P_0^{neg}$. The point is that
both $P_0^+$ and $P_0^{neg}$ are also determined from the Poincar\'e series $P_0$. Indeed,
$P_0$ has a unique
`polynomial + negative degree part' decomposition into a sum of a polynomial and a rational function of negative degree, and in fact, this decomposition is exactly
  $P_0=P^+_0+P^{neg}_0$ \cite{BN,NO1} (for the more general
   multivariable case see  \cite{LN,LNN2}.
In particular, $P_0$ determined all the integers $\{N(\ell)\}_{\ell\geq 0}$ and $p_g$ as well.

The entries $N(\ell)+1$ in the Poincar\'e series suggest to define
for any Seifert rational homology sphere the set
\begin{equation}\label{eq:defM}
\mathcal{M}=\mathcal{M}_M:=\{\ell\in \mathbb{Z} \ | \ 1+N(\ell)\geq 0\}.
\end{equation}
Clearly $\mathcal{S}_M\subset \mathcal{M}_M$. Notice that, in general,  $\mathcal{M}$ is not a numerical semigroup, consider eg. the Seifert manifold $\Sigma(2,3,7)$ with Seifert invariants
 $(-1; (2,1),(3,1),(7,1))$, when $-1\in \mathcal{M}$ but $-84\not\in \mathcal{M}$.
 However, it is a module over the semigroup $\mathcal{S}_M$ (if $\ell_1\in \mathcal{S}_M$ and
 $\ell_2\in \mathcal{M}_M$ then $\ell_1+\ell_2\in \mathcal{M}_M$).
Again by the definition of $N(\ell)$ one has
$\ell\in\mathcal{M}_M$ for any $\ell\geq (d-1)/|e|$ and $N(\ell)<-1$ for $\ell\ll 0$. 
Hence, the set  $\mathbb{Z}\setminus\mathcal{M}$ is bounded from above
and we can define the Frobenius number $f_{\mathcal{M}}$ of $\mathcal{M}$ as the largest integer not in $\mathcal{M}$. Similarly, $\mathcal{M}$ is bounded from below, thus one can define $\min\{\mathcal{M}\}$ to be the smallest integer in the module $\mathcal{M}$.

If $N(\ell)=-1$ then $\ell$ is neither in the support of $P_0(t)$ nor in the support of
$P_0^+(t)$. However, this set (non--coded by these series) is important (eg. from the point of view of
sharp cohomology vanishing). It is exactly the difference between the supports of $\mathcal{M}_M$ and
$\mathcal{S}_M$.

Furthermore, $\mathbb{Z}_{\geq 0}\setminus \mathcal{M}$ is exactly the support of $P_0^+(t)$,
hence 
\begin{equation}\label{rational}
(X,o)\ \mbox{is rational} \ \Leftrightarrow \ p_g=0\ \Leftrightarrow\ 
P_0^+\equiv0\ \Leftrightarrow
\ \mathbb{Z}_{\geq 0}\subset  \mathcal{M}.
\end{equation}

\subsubsection{The combinatorial number $\gamma$}\label{ss:gamma}
Before we recall certain known facts regarding the integers $\{N(\ell)\}_{\ell}$
(and certain consequences regarding $\mathcal{S}_M$ and $\mathcal{M}_M$)
we define a key numerical invariant:
\begin{equation}\label{eq:gamma}
\gamma:=\frac{1}{|e|}\cdot \Big( d-2-\sum_{i=1}^d \frac{1}{\alpha_i}\Big)\in \mathbb{Q},
\end{equation}
which has a central importance  regarding properties of
weighted homogeneous surface singularities or Seifert rational homology spheres.
It has several interpretations.

First of all, the adjunction formula gives (in any graph) the identity
\begin{equation}\label{eq:ZK}
Z_K=E+\sum_{v}(\delta_v-2)E^*_v,
\end{equation}
where  $E:=\sum_vE_v$. In a star--shaped graph the $E_0$ coefficients of all
 $E^*_v$ associated with end--vertices   are computed by $-(E_v^*,E^*_0)=1/(|e|\alpha_v)$ and the
 $E_0$ coefficients of
 $E^*_0$ is  $-(E_0^*,E^*_0)=1/|e|$
 (cf. \cite[(11.1)]{NOSZ}, or the determinant--formula  from \ref{ss:nsstop}).
 Hence,  the $E_0$-coefficient of $Z_K$ is exactly $\gamma+1$.

The number $-\gamma$ is called the `log discrepancy' of $E_0$, $\gamma$ the `exponent' of the weighted
homogeneous germ $(X,o)$, and $\mathfrak{o}\gamma$ is the Goto--Watanabe $a$--invariant (cf. \cite[(3.1.4)]{GW}) of the universal abelian
covering of $(X,o)$.
In \cite{neumann.abel} $e\gamma$ appears as the orbifold Euler characteristic (see also
\cite[3.3.6]{NOem}).

\subsubsection{More on $Z_K$}\label{sss:more}

If $(X,o)$ is $ADE$, then $Z_K=0$ and $\gamma=-1$.
By \cite{CR}, $\gamma|e|=d-2-\sum_i 1/\alpha_i$ is negative if and only of $\pi_1(M)$ is finite (this
can happen only if $d=3$ and $\sum_i1/\alpha_i>1$). In this case, $(X,o)$ is a quotient singularity,
hence rational.  Thus, in this case $P^+_0(t)\equiv 0$ too.

Therefore, if $(X,o)$ is not rational then $\gamma\geq 0$, that is, the $E_0$--coefficient of $Z_K$ is
$\geq 1$.
We claim that in these cases all the coefficients of $Z_K$ are strict positive.
(Indeed, $Z_K$ restricted on any leg, say $Z_K|_i$,  by adjunction relations satisfies
$(Z_K|_i, E_{ik})\leq 0$, $1\leq k\leq \nu_i$, with strict inequality for $k=1$, hence
all the coefficients of $Z_K|_i$ are strict positive.) In particular, we also have $Z_K\geq
r_{[Z_K]}$.

Let us rename  the $i^{th}$ end--vertex by $E_i$ and let us compute the $E_i$--coefficient of $Z_K$.

 For these one needs
for end--vertices $i,j$ the identities $(E_i^*,E_j^*)=(e\alpha_i\alpha_j)^{-1}$ for $i\not=j$ and
$(E_i^*,E_i^*)=(e\alpha_i^2)^{-1}-\omega'_i/\alpha_i$ if $i=j$;
 cf. \cite[(11.1)]{NOSZ}, or the determinant--formula  from \ref{ss:nsstop}. Therefore, by
 (\ref{eq:ZK}) and a computation,
 \begin{equation}\label{eq:ZKend}
 -(Z_K,E^*_i)=1+(\gamma-\omega_i')/\alpha_i.
 \end{equation}
Hence, if $\Gamma$ is numerically Gorenstein (that is, $Z_K\in L$) then
$\gamma\in \mathbb{Z}$ and $\gamma\equiv\omega_i'\ ({\rm mod}\ \alpha_i)$ for all $i$.
Hence,  in the numerically Gorenstein non--$ADE$ case --- when we already know that $\gamma\geq 0$ ---
by this congruence we get the stronger $\gamma\geq 1$. (If $\gamma=1$ and $\Gamma$ is
minimally Gorenstein
then $\omega_i'=\omega_i=1$ and $b_0=d-2$, hence $(X,o)$ is the `polygonal'
 minimally elliptic singularity with $Z_K=E+E_0$.)

\subsubsection{Some inequalities satisfied by $\{N(\ell)\}_{\ell}$.}
The following facts are proved eg. in \cite[Sections 3.1--3.3]{NOem} (some of them can be
verified by direct computations).

\begin{prop}\label{prop:INEQ} \

(a) \ $-(\alpha-1)|e|-d\leq N(\ell)-\lceil \ell/\alpha\rceil \alpha |e| \leq -1$. In particular
$\lim_{\ell\to\infty}N(\ell)=\infty$.

(b) \ If $\ell>\gamma$ then $h^1(E_0,\mathcal{O}_{E_0}(D^{(\ell)}))=0$, ie. $N(\ell)\geq -1$.
In particular, $\deg P_0^+(t)\leq \max\{0,\gamma\}$.

(c) \ If $(X,o)$ is numerically Gorenstein (but not $ADE$) then $N(\gamma)=-2$ and  $\deg P_0^+(t)= \gamma$.

(d) \ $N(\alpha)=\alpha(b_0-\sum_i\omega_i/\alpha_i)=\alpha |e|=\mathfrak{o}>0$.

(e) \ $N(\ell+\alpha)=N(\ell)+N(\alpha)=N(\ell)+\mathfrak{o}>N(\ell)$ for any $ \ell\geq 0$.

(f) \ $N(\ell)\geq 0$ for any $\ell>\alpha+\gamma$.
\end{prop}

\subsubsection{The numerically Gorenstein case}\label{sss:NGrev}

Proposition \ref{prop:INEQ} implies the following.

\begin{cor}\label{cor:frob} \  Assume that $\Gamma$ is
numerically Gorenstein non--$ADE$.

(a) Then $f_{\mathcal{M}_M}=\gamma$.

(b) Assume additionally that $\mathfrak{o}=1$ (eg. $M$ is an integral homology sphere). Then $f_{\mathcal{S}_M}=\alpha+\gamma$.
\end{cor}
\begin{proof}
{\it (a)} Use {\it (b)--(c)} of Prop. \ref{prop:INEQ}. For {\it (b)}
 use {\it (e)--(f)} of Prop. \ref{prop:INEQ} and $N(\gamma)=-2$. % and $N(\alpha+\gamma)=-2+\mathfrak{o}=-1$.
\end{proof}

The following numerical symmetry is the trace of the Gorenstein symmetry.
In the case when $M$ is an integral homology sphere the statement was
proved in  \cite{CK} too. For a generalization here see Proposition \ref{prop:chidual}.
\begin{prop}\label{prop:SYM} If the graph is numerically Gorenstein  (that is, $Z_K\in L$),
 then
\begin{equation}\label{SHSsym}
N(\ell)+N(\gamma-\ell)=-2 \ \ \mbox{for any} \ \ell\in\mathbb{Z},
\end{equation}
\end{prop}
\begin{proof}
 From (\ref{eq:ZKend}) we have  $\gamma\equiv \omega_i'\ ({\rm mod}\ \alpha_i)$,
 hence $\omega_i\gamma\equiv \omega_i\omega_i'\equiv 1\ ({\rm mod}\ \alpha_i)$ too.
 Then
 $$N(\gamma-\ell)=b_0(\gamma-\ell) -\sum_i \frac{\omega_i\gamma-1}{\alpha_i}+\sum_i \Big\lceil
 \frac{1-\omega_i\ell}{\alpha_i}\Big\rceil.$$
 Then using $\lceil (1-x)/\alpha_i\rceil+\lceil x/\alpha_i\rceil=1$ $(x\in \mathbb{Z})$
 the statement follows.
\end{proof}
%For a generalization  see Proposition   \ref{prop:Nsym}.

This combined with Corollary \ref{cor:frob}{\it (b)} gives
\begin{cor}\label{cor:SMsym}
 If the graph is numerically Gorenstein non-$ADE$ then
\begin{equation}\label{eq:SymSM}\begin{split}
\ell\in \mathcal{M}_M \ &\Leftrightarrow \ \gamma-\ell\not \in \mathcal{S}_M\\
\{\ell\,:\, N(\ell)=-1\}&= \{\ell\,:\, \gamma-\ell\not\in\mathcal{S}_M,
\ \ell\not\in\mathcal{S}_M\}=\mathbb{Z}\setminus \Big((\gamma-\mathcal{S}_M)\cup\mathcal{S}_M\Big).
\end{split}\end{equation}
In particular, $\min\{\mathcal{M}_M\}+f_{\mathcal{S}_M}=\gamma$.
This,  via Corollary \ref{cor:frob}{\it (b)} for
an integral homology sphere $M$ reads as $\min\{\mathcal{M}_M\}=-\alpha$.
\end{cor}
(Let us verify `elegantly',  using the first identity above,  that $\mathcal{M}$ is an
$\mathcal{S}$--module. Since $\mathcal{S}$ is a semigroup one has
$(\gamma-\ell\not\in \mathcal{S}, \ s\in\mathcal{S}) \Rightarrow \gamma-\ell-s\not\in \mathcal{S}$.
This transforms into $(\ell\in \mathcal{M}, \ s\in\mathcal{S}) \Rightarrow \ell+s\in \mathcal{M}$.)

\subsubsection{Connection with Gorenstein analytic type}\label{sss:Gortype}
 The identity (\ref{SHSsym}) is the topological `trace' of Gorenstein duality of algebraic geometry.
 Indeed, a weighted homogeneous singularity with numerically Gorenstein topological type
 is automatically Gorenstein. In this case one has
 $D^{(\ell)}+D^{(\gamma-\ell)}=K_{E_0}$ in ${\rm Pic}(E_0)$ (even if $g$ is not zero),
 which at ${\rm deg}$--level (with $g=0$) is (\ref{SHSsym}).
 The Gorenstein duality implies that $h^1(\mathcal{O}_{E_0}(D^{(\ell)}))=
 h^0(\mathcal{O}_{E_0}(D^{(\gamma-\ell)}))$. If $g=0$ then the very same formula follows
 already from (\ref{SHSsym}), since $\max\{0, -N(\ell)-1\}=\max\{0, N(\gamma-\ell)+1\}$.

This, for any numerically Gorenstein Seifert rationally homology sphere reads as
$$P_0^+(t)=\sum_{\ell=0}^\gamma \dim R_{X,\gamma-\ell}t^\ell=
\sum_{\ell=0}^\gamma \dim R_{X,\ell}t^{\gamma-\ell}.$$

\subsubsection{The connection between $\mathcal{S}_{an}$ and $\mathcal{S}_M$}\label{ss:swhII}
Above we defined two semigroups (monoids), $\mathcal{S}_{an}$ is a submonoid of
$L_{\geq 0}=\mathbb{N}^{|\mathcal{V}|}$,
sitting in the Lipman cone $\mathcal{S}_{top}$. The other is a numerical submonoid of $\mathbf {N}$.
The next statement basically follows from \cite{Pinkham}.
\begin{lemma}\label{lem:twoS}
The elements of $\mathcal{S}_M$ are the $E_0$--coefficients of the elements of
$\mathcal{S}_{an}$.
\end{lemma}
\begin{proof}
Associated with the irreducible component $E_{0}$ of the exceptional divisor we consider the divisorial filtration $\mathcal{F}:=\{\mathcal{F}(\ell)\}_{\ell\in\mathbb{Z}}$, where
$\mathcal{F}(\ell):=\{f\in \mathcal{O}_{(X,0)}| (f)_{\Gamma}|_{E_{0}}\geq \ell \}$
(and $(f)_{\Gamma}|_{E_{0}}:=((f)_{\Gamma},-E^*_{0})$ is the $E_{0}$-coefficient of the principal cycle $(f)_{\Gamma}$ associated with $f \in  \mathcal{O}_{(X,0)}$, cf.  \ref{ss:princc}).
On the other hand, by \cite{Pinkham}, $\mathcal{F}(\ell)=\oplus _{\ell'\geq \ell}
R_{X,\ell'}$. Hence $R_{X,\ell}=\mathcal{F}(\ell)/\mathcal{F}(\ell+1)$. Hence $\ell\in \mathcal{S}_M$
if and only if there exists $(f)_{\Gamma}\in\mathcal{S}_{an}$ with
$(f)_{\Gamma}|_{E_{0}}=\ell$.
\end{proof}
This fact shows that there is a richer monoid behind $\mathcal{S}_M$, namely $\mathcal{S}_{an}$.

\subsubsection{The combinatorial determination of $\mathcal{S}_{an}$.}\label{sss:Sanformula}
Before we state the next result we observe that both monoids $\mathcal{S}_{top}$ and
$\mathcal{S}_{an}$ are closed to taking the $`\min' $  in $L\subset \mathbb{Z}^{|\mathcal{V}|}$:
namely, if $l_1, l_2\in \mathcal{S}$ then $\min\{l_1,l_2\}\in\mathcal{S}$ too.

The next theorem was proved for
`splice-quotient singularities' \cite[Theorem 7.1.2]{Nsplice}. Since
 weighted homogeneous singularities are splice quotient, the theorem applies in our situation.
 Below we formulate it tailored already to the weighted homogeneous
 rational homology sphere case.

  Let $\mathcal{E}$ be the set of end--vertices. An {\it integral monomial cycle}
  has the form $\mathfrak{M}:=\sum_{e\in \mathcal{E}}n_eE^*_e\in L$, where each $n_e\in \mathbb{Z}_{\geq 0}$.
   (Note that the cycles $E^*_e$ usually are {\it rational} cycles, hence the condition that
   $\mathfrak{M}$ is integral imposes some `sub--lattice' condition on the coefficients $\{n_e\}_e$
   whenever $H$ is non--trivial.)
 \begin{thm}\label{th:San}\ \cite[Theorem 7.1.2]{Nsplice}
 Any element of $\mathcal{S}_{an}$ can be obtained as $\min_k\{\mathfrak{M}_k\}$ 
 for certain finitely many
 integral monomial cycles $\{\mathfrak{M}_k\}_k$.
 \end{thm}
% , one has the following characterization by
%$$l\in \mathcal{S}_{an} \ \mbox{ if and only if } \ l=\inf_k \ \sum_{v\in\mathcal{E}} c_v^{(k)}E_v^*
%\ \ \mbox{for a finite set} \ \{\sum_{v\in\mathcal{E}} c_v^{(k)}E_v^*\}_k.$$

\section{Strongly flat semigroups and  Seifert homology spheres}\label{s:Ssf}

\subsection{} In this section let $M=\Sigma(\alpha_1,\dots,\alpha_d)$ be a Seifert
integral homology sphere. Thus, $\alpha_1,\ldots,\alpha_d \geq 2$ are pairwise relatively prime integers and $b_0$, $(\omega_1,\ldots,\omega_d)$ are uniquely determined by the Diophantine equation $\alpha(b_0-\sum_{i=1}^d\omega_i/\alpha_i)=1$.

First note that if we set  $a_i:=\alpha/\alpha_i$
then the greatest common divisor of $a_1, \ldots, a_{i-1}, a_{i+1},\ldots, a_d$
 is $\alpha_i$, hence the system $\{a_i\}_{i=1}^d$ generates a strongly flat semigroup
 $G(a_1, \ldots, a_d)$. Moreover, using the above Diophantine equation one shows that
 $N(a_i)=0$. Hence $G(a_1, \ldots, a_d)\subset \mathcal{S}_M$. This fact was already noticed by
  Can and Karakurt in \cite[Theorem 4.1]{CK} too.

Let us provide the analytic interpretation/proof  of this statement.

A weighted homogeneous  analytic realization is given by a Brieskorn
 isolated complete intersection. This consists of $d-2$ equations in $(\mathbb{C}^d,0)$ of type
 $\sum _{i=1}^d a_{ki} z_i^{\alpha_i}=0$ $(k=1,\ldots,d-2)$, where the matrix
 $\{a_{ki}\}_{i,k}$ has full rank.

 The strict transforms of the coordinate functions $z_i$ are supported on the end--exceptional curves,
 their divisors $(z_i)_\Gamma$ are exactly the cycles $E^*_i$ associated with the end--vertices.
 In particular, these cycles belong to $\mathcal{S}_{an}$. 
($N(a_i)=0$ shows that $\dim (R_{X,a_i})=1$, that is, each $R_{X,a_i}$ is a 1--dimensional vector space
generated by the coordinate function $z_i$.)
 The $E_0$--coefficient of $E_i^*$ is $a_i$ (cf. \cite{NOSZ}). Hence the  inclusion
 $G(a_1, \ldots, a_d)\subset \mathcal{S}_M$
 follows  from   Lemma \ref{lem:twoS} as well.

  Can and Karakurt in \cite[Theorem 4.1]{CK} observed also
  that the two numerical semigroups on the interval $[0,\gamma]$ agree.
Thank to Theorem \ref{th:San}   we show that in fact they agree everywhere.

\begin{thm}\label{thm:Sgen}
$\mathcal{S}_M=G(a_1, \ldots, a_d)$,
that is, the strongly flat semigroups are exactly the numerical semigroups associated with
Seifert integral homology spheres $\Sigma(\alpha_1,\dots,\alpha_d)$.
\end{thm}
\begin{proof}
We will apply Lemma \ref{lem:twoS} and Theorem \ref{th:San}. 
By Theorem \ref{th:San}
the cycles $E^*_i$ generate $\mathcal{S}_{an}$ by the operations $+$ and $\min$.
Note that the $E_0$--coefficient of
$\min_k\{\mathfrak{M}_k\}$ is the $E_0$ coefficient of one of the integral monomial cycles 
$\mathfrak{M}_k$. On the other hand, since $H=L'/L=0$, all the  dual cycles $E^*_i$ are automatically integral, hence $\mathcal{S}_M$ is generated by the $E_0$--coefficients of the end dual cycles $E^*_i$, that is, by the $a_i$'s.
\end{proof}

In this case ${\rm lcm}(a_1,\ldots, a_d)=\alpha$, hence the right hand side of (\ref{fbound})
(via (\ref{eq:sei2}) is $\gamma+\alpha$. Hence the two statements 
$f_{\mathcal{S}_M}=\gamma+\alpha$
and $ f_{G(a_1,\ldots,a_n)}=\gamma+\alpha$, 
from  Raczunas and Chrz\c astowski-Wachtel \cite{RChW} (cf. \ref{ss:DFP}) and from Corollary
\ref{cor:frob}{\it (b)} respectively,  identify.

\subsubsection{The symmetry of $\mathcal{S}_M$}
It is really amazing how naturally appears another symmetry
in this integral homology sphere  situation,
at this time of $\mathcal{S}_M$. 
First notice that 
 \begin{equation}\label{eq:MIHS}
 f_{\mathcal{M}_M}=\gamma \ \  \mbox{and} \ \  \min\{\mathcal{M}_M\}=-\alpha.\end{equation}
If  $\Gamma\not=-E_8$, then it follows  from Corollaries \ref{cor:frob}{\it (a)} and
\ref{cor:SMsym}. In the rational case (that is, for $-E_8$) $f_{\mathcal{M}}=\gamma=-1$, and 
$f_{\mathcal{S}_M}=\gamma+\alpha$ is still valid (use eg. (\ref{fbound})), and $\min\{\mathcal{M}_M\}=-\alpha$ too (use Proposition \ref{prop:SYM}).

However, one can do much more. Since $N(\alpha)=\mathfrak{o}=1$, by Proposition \ref{prop:INEQ}{\it (e)}
$N(\alpha+\gamma-\ell)=N(\gamma-\ell)+1$, hence (\ref{SHSsym}) reads as
\begin{equation}\label{SHSsym2}
N(\ell)+N(\alpha + \gamma-\ell)=-1 \ \ \mbox{for any} \ \ell\in\mathbb{Z}.
\end{equation}
this implies (compare with  (\ref{eq:SymSM}))
\begin{equation}\label{eq:SymSS}
\ell\in \mathcal{S}_M \ \Leftrightarrow \ \alpha+ \gamma-\ell\not \in \mathcal{S}_M.
\end{equation}
In  particular, $\mathcal{S}_M$ is symmetric. Moreover,
one verifies (either by a combination of
(\ref{eq:SymSM}) and (\ref{eq:SymSS}), or by
$(N(\ell)\geq  0)\Leftrightarrow (N(-\alpha+\ell)+1\geq 0)$ via Proposition  \ref{prop:INEQ}{\it (e)})
that
\begin{equation}\label{eq:MGor}
\mbox{$\mathcal{M}_M$ is generated as $\mathcal{S}_M$--module by \underline{one}
 element, namely by $-\alpha=\min\{\mathcal{M}_M\}$.}
\end{equation} In other words, $\mathcal{M}_M=-\alpha+\mathcal{S}_M$.

\subsubsection{} Lemma \ref{lem:twoS} and Theorem \ref{thm:Sgen}
have the following important messages:

(a) The combinatorial approach of the   theory of strongly flat semigroups is replaced by
topological and algebro-geometrical tools and methods. This connection also gives the
possibility of wider applicability of the classical combinatorial theory in topology and singularity theory (and viceversa). It would be interesting to find a similar topological
realizability for any $G(a_1,\ldots, a_d)$.

(b) As the strongly flat $G(a_1,\ldots, a_d)$ is the one--coordinate projection of an affine monoid $\mathcal{S}_{an}$, one should try to find similar embedding of arbitrary numerical
semigroups  into some canonical affine monoids,
which hopefully carry more geometry (and definitely their study is more challenging).

 (c) As  $G(a_1,\ldots, a_d)$ is the $E_0$--projection of $\mathcal{S}_{an}$, one can ask,
 what about the other projection. Are they `interesting' numerical semigroups?

 Eg., let us determine the numerical semigroup $\mathcal{S}_{end}$
 obtained from $\mathcal{S}_{an}$ associated with $\Sigma(\alpha_1,\ldots, \alpha_d)$ projected on the
 end--vertex associated with index $d$.

 Again, $\mathcal{S}_{an}$ is generated by the duals of the end--vertices $\{E^*_i\}_{i=1}^d$.
 Their $E_d$--coefficients were already listed in front of (\ref{eq:ZKend}).
 Namely, the $E_d$ coefficient of $E^*_i$ $(1\leq i<d)$ is $1/(|e|\alpha_i\alpha_d)=\alpha /
( \alpha_i\alpha_d)$. Hence the first $(d-1)$ generators generate the  usual
$G$--type monoid associated with $\alpha_1, \ldots, \alpha_{d-1}$. On the other hand, the
$E_d$--coefficient of $E_d^*$ is $1/(|e|\alpha_d^2)+\omega_d'/\alpha_d$, which after a computation
(and using the Diophantine equation) is $\lceil \alpha_1\cdots \alpha_{d-1}/\alpha_d\rceil$.
Hence, $\mathcal{S}_{end}$ is generated by
$$\alpha_2\cdots\alpha_{d-1}, \ \alpha_1\alpha_3\cdots\alpha_{d-1},\cdots, \
\alpha_1\cdots \alpha_{d-2},\
\lceil \alpha_1\cdots \alpha_{d-1}/\alpha_d\rceil,
$$
where all $\alpha_i$ are pairwise relative prime.

Eg., if $\alpha_d$ is very large compared with the others then the last generator is 1,
hence merely this generator generates $\mathcal{S}_{end}$, which is $\mathbb{N}$.
There are choices when the last generator is already in the monoid generated by the others, hence, as generator can be deleted. And there are also examples when this last generator is the
conductor of the  monoid generated by  the first $d-1$ generators (take eg. $\Sigma(2,3,7)$).
So, in general, the computation   of the minimal set of generators, and also of the
Frobenius number,  can be a challenge.

(d)
The aim of the forthcoming sections is to generalize the above results to the case of Seifert rational homology spheres:  we will determine
the Frobenius numbers of  $\mathcal{S}_M$ and of the
$\mathcal{S}_M$--module $\mathcal{M}_M$.
In fact, the module will play a crucial auxiliary role: first we compute its Frobenius number, and then
we reduce the computation of $f_{\mathcal{S}}$ to $f_{\mathcal{M}}$ associated with another
aide graph.

However, it will turn out that the
 tools are far to be merely arithmetical, they rely on the mathematical machinery
  of `generalized Laufer computation sequences', which originally was used in the computation
  of analytic invariants of surface singularities, and later in the computation of the
  lattice cohomology   of their links. In order to make the presentation more complete,
  in the following we review some motivations, definitions and facts  from \cite{NOSZ,LNred}
  about the lattice cohomology
   and then we present the  special computation sequences needed in the proof.

\section{Topology and arithmetics of the quasi-linear function}\label{s:topology}

\subsection{$\{N(\ell)\}_{\ell}$ and Heegaard--Floer homology}\label{ss:MotHF}
As a preamble, let us try first to present only in a few words how  the sequence
 {$\{N(\ell)\}_{\ell}$ from (\ref{defN}) appears in modern low--dimensional topology.

Heegaard--Floer homology by Ozsv\'ath and Szab\'o \cite{OSz1} is one of the most important and highlighted invariants of $3$-manifolds. In the case of negative definite Seifert 3--manifolds the results of \cite{OSz2} and \cite{NOSZ} have shown that the calculation of Heegaard--Floer homology is purely combinatorial. In fact, it is isomorphic with the lattice cohomology introduced by
the second author  \cite{Nlat}
(see also \cite{Ngr}),  which provides a combinatorial recipe for its computation.
(In fact, everything generalizes to the case of `almost rational' graphs, see \cite{NOSZ,Ngr}.)

To start with,  we define  the discrete function $\tau:\mathbb{N}\rightarrow \mathbb{Z}$
 by the recurrence
\begin{equation}\label{eq:tau}\tau(\ell+1)=\tau(\ell)+1+N(\ell)\end{equation}
 with initial condition $\tau(0)=0$. The statement is  that
 this function (in fact,
  already its `local minimum and maximum points and values')
   determines the Heegaard--Floer homology of $M$.

It is a mystery that this topologically motivated `difference $\tau$--function'
is exactly the quasi-linear function $\ell\mapsto 1+N(\ell)$ imposed by the (analytic)
Poincar\'e series of the local ring, determined by Pinkham's theory.
For more, see sections \ref{ss:lattice}--\ref{ss:concat}, especially Remark \ref{rem:xt}.

This theoretical result initiated several work regarding the concrete properties of the function $N(\ell)$, see e.g. the work of Can and Karakurt in \cite{CK}
in the integral homology sphere case regarding the local extrema of $\tau$.
Their results have been extended in \cite{KL} and applied to the `botany problem' for Seifert homology spheres. In this work further deep structural connections with Heegaard--Floer theory are
established.

\subsection{Lattice cohomology and generalized Laufer computation sequences}\label{ss:lattice} \
The statements of this subsection will not be  used later, they have purely informative and
motivational role.

\subsubsection{Definition and reduction}\label{ss:defLH}
To define the lattice cohomology one first  needs  a lattice $\mathbb{Z}^n$
 with fixed base elements $\{E_i\}_{i=1}^n$,
which provide a cubical decomposition of $\mathbb{R}^n=\mathbb{Z}^n\otimes\mathbb{R}$:
the $0$-cubes are the lattice points $l\in \mathbb{Z}^n$, a $q$-cube will be denoted by $(l,I)$ for
some $I\subset\{1,\dots,n\}$ with $|I|=q$ and it is determined by its vertices $l+\sum_{j\in J}E_j$
for any $J\subset I$. One also considers a weight function
$w:\mathbb{Z}^n\rightarrow\mathbb{Z}$
bounded bellow, and for each cube $(l,I)$ one defines $w(l,I):=\max\{w(v)\ : \ v \ \mbox{vertex of } (l,I)\}$. Then for each integer $N\geq \min w$
one defines the simplicial complex $S_N$ as the union of all cubes with $w(l,I)\leq N$.
The lattice cohomology associated with the pair $(\mathbb{Z}^n,w)$
  is the $\mathbb{Z}[U]$-modules $\{\mathbb{H}^q(\mathbb{Z}^n,w)\}_{q\geq 0}$ defined by $\mathbb{H}^q(\mathbb{Z}^n,w):=\bigoplus_{N\geq\min w}H^q(S_N,\mathbb{Z})$.
 The $U$-action consists of the restriction maps induced by the inclusions $S_N\hookrightarrow S_{N+1}$.

In some cases one takes only part of the lattice, say $(\mathbb{Z}_{\geq 0} )^n $, and only those
cubes which are supported by this part, a weight function
 $w:(\mathbb{Z}_{\geq 0})^n\rightarrow\mathbb{Z}$, and one defines in the same way
 $\mathbb{H}^q((\mathbb{Z}_{\geq 0})^n,w)$.

In the case of normal surface singularities, we fix a good resolution $\Gamma$ and construct the lattice cohomology using the lattice $L\cong \mathbb{Z}^{|\mathcal{V}|}$ and the weight function $\chi_{k_r}$ associated with any characteristic orbit $[k]$,
usually denoted by $\{\mathbb{H}^*(L,[k])\}_{[k]}$.
In fact, it it independent of the choice of the resolution, it depends only on the link
and the $spin^c$--structure  codified by $[k]$.
 It serves as a cohomological categorification for the Seiberg--Witten invariant, see eg. \cite{NOSZ,NJEMS}.

 The Reduction theorem of \cite{LNred} allows to reduce the rank of the lattice
and evidences the complexity of the cohomology. In fact, in the case of  Seifert rational homology
spheres (and in general, in the case of plumbed manifolds associated with `almost rational'
 graphs) the higher cohomology $\mathbb{H}^{\geq 1}(L,[k])=0$, while $\mathbb{H}^{0}(L,[k])$
is reduced to the rank--one case
$\mathbb{H}^0(\mathbb{Z}_{\geq 0},w_{[k]})$, where $\ell\mapsto w_{[k]}(\ell)$ ($\ell\geq 0$)
is determined via  some special universal cycles, denoted by  $\{x_{[k]}(\ell)\}_{\ell}$  in
\cite{LNred}, \cite{NOSZ}, and via  $w_{[k]}(\ell)=\chi_{k_r}(x_{[k]}(\ell))$.
If $[k]=[K]$, then $w_{[k]}(\ell)=\tau(\ell)$ from (\ref{eq:tau}), see Remark \ref{rem:xt}.

 We will not provide here any cohomology computation,
 however by the above  comment we  wished to inform the reader about the
 origins of the technical tools from below, and why were they  developed. The interested reader might
 consult the articles mentioned above and the references therein.

\subsubsection{Generalized Laufer computation sequences}\label{ss:compseq} \
%connecting $r_{[Z_K]}$ and $s_{[Z_K]}$} \

One of
the main technical tools of the computation of the lattice cohomology is the `computation sequence'.
A computation sequence has the form $\{z_i\}_{i=0}^k$, $z_i\in L'$, such that for any
$i$ one determines by some algorithmic rule a vertex $v(i)\in\mathcal{V}$ using which one sets
$z_{i+1}:=z_i+E_{v(i)}$. It is the discrete latticial analogue of a continuous path.
It connects $z_0$ to $z_k$. The cycle
$z_k$ usually has some universal property targeted by the algorithm.

As we already mentioned, we have to compute the simplicial cohomology of the simplicial complexes
$S_N$ (a union of latticial cubes). Note also that on $S_N$ we have the restriction of the weight
function $\chi_{k_r}$. Recall that in Morse theory, in the presence of a height function $f$,
we wish to contract our space along the flow determined by $f$ (by decreasing $f$),
or to find extremal points as universal points at the `end' of flow--lines. In our case we
proceed similarly, the computation sequences are the `flow--pathes' determined by an initial
lattice point and the weight (hight) function $\chi_{k_r}$. In this way we can contract
$S_N$, or we find  some `extremal' lattice points with universal properties.

The first computation sequence in $L$ was introduced by Laufer by a special
optimization algorithm (flowing in the direction of the Lipman cone $\mathcal{S}_{top}$),
it is $\chi$--decreasing,  starts with $E$ and ends at $Z_{min}$,
the Artin's cycle, the minimal non--zero element of the Lipman cone $\mathcal{S}_{top}$ \cite{Laufer72}.
In this way Laufer
succeeded to define a rationality criterion (basically equivalent with the contractibility
of all spaces $S_N$, see also \cite{NOSZ,Ngr}).
The tool was successfully used later in the theory of elliptic singularities as well
\cite{Laufer77,weakly,OkumaEll}. For application in lattice cohomology see \cite{NOSZ,Ngr,LNred}.

\subsubsection{}\label{sss:shift}
 Note that the restriction of $\chi_{k_r}$ on $L$ (where $k_r=K+2s_h$)
can be rewritten as the restriction of $\chi$ on the shifted lattice $s_h+L=r_h+L$.
Indeed, for any $l\in L$ one has $\chi_{k_r}(l)=\chi(l+s_h)-\chi(s_h)$. Below in our computations we will use this second approach. In fact, we will not need all the possible classes, but only those
associated with $h=0$ and $h=[Z_k]$ (if $\Gamma$ is numerically Gorenstein, then even these two classes
agree). For them the associated package of invariants will be compared by duality.

\subsection{Universal cycles and their  concatenated sequences}\label{ss:concat}
The next lemma   generalizes similar statements form
 \cite[\S 7]{NOSZ}, \cite[Prop. 4.3.3]{Ngr} and \cite[\S 3]{LNred}, see also \cite{Laufer72}.

\begin{lemma}\label{lem:cs} \
Set $\mathcal{V}^*:=\mathcal{V}\setminus\{v_0\}$, and take $r_h$ associated with $h\in H$ as above.
%\note{talan ez megvan, nem? ??Prop4.33 -Nemethi-Gradroots??}
\begin{enumerate}
 \item[(1)] For any $\ell\in \mathbb{Z}_{\geq 0}$ there exists a unique minimal (rational) cycle
  $x^{\ell}$ with respect to the following properties:
 \begin{itemize}
  \item[(a)] $m_0(l')=m_0(r_{h})+\ell$ and
  \item[(b)] $l'\in (r_{h}+L)$ and $(l',E_v)\leq 0$ for any $v\in \mathcal{V}^*$.
 \end{itemize}
\item[(2)]  One has  $x^0\geq r_h$ and $x^{\ell+1}\geq x^{\ell}+E_0$ for any $\ell\geq 0$.
 \item[(3)] For any $l'\in r_{h}+L$ satisfying (a) we have $\chi(l')\geq \chi(x^\ell)$.
 \item[(4)] There is a `generalized Laufer computation sequence' $\{z_i\}_{i\geq 0}$ connecting $x^\ell$ and $x^{\ell+1}$ as follows: set $z_0:=x^\ell$, $z_1:=x^\ell+E_0$ and assume that $z_i$ ($i\geq 1$) is already constructed. If $z_i$ does not satisfy (b), then we have $(z_i,E_{v(i)})>0$ for some $v(i)\in\mathcal{V}^*$ and set $z_{i+1}=z_i+E_{v(i)}$. Otherwise, we necessarily have $z_i=x^{\ell+1}$.

     Similarly, there is a `generalized Laufer computation sequence' $\{z_i\}_{i\geq 1}$
     connecting $r_h$ and $x^{0}$ as follows: set  $z_1:=r_h$, and then repeat the definition
     from the previous case.
\item[(5)] $\chi(x^0)=\chi(r_h)$ and $\chi(x^{\ell+1})-\chi(x^\ell)=1-(x^\ell,E_0)$ for $\ell\geq 0$.
\end{enumerate}
\end{lemma}
\begin{proof}
{\it (1)} The proof is similar to one given by Laufer targeting $Z_{min}$ in \cite{Laufer72}.
The negative definiteness of the intersection form guarantees
 the existence of a cycle with properties (a) and (b).  Then, it is enough to prove that
 if $x_1,x_2$ satisfy (a) and (b) then $x:=\min\{x_1,x_2\}$ also satisfies (a) and (b). Indeed, for any $v\in \mathcal{V}^*$ there is at least one index $i\in\{1,2\}$ such that $E_v\notin |x_i-x|$. Hence, $(x,E_v)=(x_i,E_v)+(x-x_i,E_v)\leq 0$.\\
{\it (2)}
Assume that
we can write $x^\ell=x^\ell_1-x^\ell_2$ such that $x^\ell_i\geq 0$ for $i\in \{1,2\}$ and $|x^\ell_1|\cap|x^\ell_2|=\emptyset$. Then for any $E_v \in |x^\ell_2|$ we have $(-x^\ell_2,E_v)\leq (x^\ell,E_v)\leq 0$. In particular, $(x^\ell_2)^2\geq 0$ which implies $x^\ell_2=0$ by the negative definiteness of the intersection form. Since $x^\ell\in r_h+L$ we get $x^\ell\geq r_h$.

Finally observe that  $x^{\ell+1}-E_0$ satisfy (a)--(b) for $\ell$, hence by the minimality
of $x^\ell$ the last inequality follows as well.
\\
{\it (3)} Assume first that $l'\leq x^\ell$. Then there is a generalized Laufer computation sequence $\{z_i\}_{i=0}^m$ connecting $l'$ with $x^\ell$ as follows: set $z_0:=l'$ and assume that $z_i$ is already constructed. If for some $v(i)\in \mathcal{V}^*$ one has $(z_i,E_{v(i)})>0$ continue with $z_{i+1}=z_i+E_{v(i)}$.  We claim that along these steps still $z_i\leq x^\ell$. We verify this by
induction, we assume that $z_i\leq x^\ell$ and we prove that $z_{i+1}\leq x^\ell$.
For this observe that $m_{v(i)}(z_i)=m_{v(i)}(x^\ell)$ cannot happen. Indeed, in such case,
$(x^\ell, E_{v(i)})=(x^\ell-z_i, E_{v(i)})+(z_i, E_{v(i)})>0$, a contradiction. Hence
necessarily $m_{v(i)}(z_i)<m_{v(i)}(x^\ell)$, ie. $z_{i+1}\leq x^\ell$ too.
In particular, the constructed sequence $\{z_i\}_i$ must stop, say at $z_m$, and
$z_m\leq x^\ell$. Then, by the minimality of $x^\ell$ from part {\it (1)} we have $z_m=x^\ell$.
Moreover, along the computation sequence $\chi(z_{i+1})=\chi(z_i)+1-(z_i,E_{v(i)})\leq \chi(z_i)$
for any $0\leq i< m$. Hence, $\chi(l')\geq \chi(x^\ell)$.

In general, we write $l'=x^\ell-l_1+l_2$ such that $l_1,l_2\in L$ with $l_1\geq 0$, $l_2\geq 0$, both are supported on $\mathcal{V}^*$ and $|l_1|\cap|l_2|=\emptyset$. Then $\chi(l')=\chi(x^\ell-l_1)+\chi(l_2)+(l_1,l_2)-(x^\ell,l_2)$. Since $(l_1,l_2)\geq 0$ by their supports, $-(x^\ell,l_2)\geq 0$ by (b), and $\chi(l_2)\geq 0$ by the fact that it is supported on a rational subgraph (cf. \cite{Artin62,Artin66}), we get $\chi(l')\geq \chi(x^\ell-l_1)$. On the other hand, we have proved $\chi(x^\ell-l_1)\geq \chi(x^\ell)$. \\
{\it (4)} Use the inequalities from {\it (2)} and the
computation sequence from the proof of {\it (3)}.\\
{\it (5)} One verifies that the computation sequence from {\it (4)} is $\chi$--constant
starting from $i=1$ (use
Laufer's rationality criterion, for details see \cite{NOSZ}). Hence
$\chi(x^{\ell+1})-\chi(x^\ell)=\chi(x^\ell+E_0)-\chi(x^\ell)$.
\end{proof}
The cycles $x^\ell$ depend on $h$ (though we did not emphasize it notationally).
We will adopt the following notations: if $h=0$ then we write $x(\ell):=x^{\ell}$ (associated with
$h=0$) and for $h=[Z_K]$ we set $x^*(\ell):=x^{\ell}$ (associated with  $h=[Z_K]$).

\begin{remark}\label{rem:xt}
For $h=0$ one verifies (see \cite[Prop. 11.11]{NOSZ}) that
$$x(\ell)=\ell E_0+\sum _{i=1}^d\Big\lceil \frac{\ell\omega_i}{\alpha_i}\Big\rceil E_{i1}+
\mbox{terms supported on other base
elements}.$$
Hence $\chi(x(\ell+1))-\chi(x(\ell))=1-(x(\ell),E_0)=1+N(\ell)$.
\end{remark}

\begin{remark} Regarding the discussion from subsection \ref{ss:lattice} we add the following.
For $h=0$ the cycles $\{x(\ell)\}_{\ell\geq 0}$ are exactly the cycles which
determine the reduced (`half') lattice $(\mathbb{Z}_{\geq 0},w)$,
by $w(\ell)=\chi(x(\ell))=\tau(\ell)$. If $h\not=0$ then the cycles $x^{\ell}$ and $x_{[k]}(\ell)$
(mentioned in section \ref{ss:lattice})
and the corresponding weight functions are
connected by a shift (given by the identity from \ref{sss:shift}).

For the description of cycles $x^\ell$ for any $h$ see \cite[Prop. 11.11]{NOSZ} again.
\end{remark}

\subsubsection{The computation sequence targeting $\mathcal{S}'$}\label{sss:s}
There is another/similar computation sequence (valid for any connected negative definite graph), which starts with any element and ends in $\mathcal{S}'$.
\begin{lemma}\label{lem:cs2} \ Fix some $h\in H$ and $l'\in r_h+L$.
\begin{enumerate}
 \item[(1)] There exists a unique minimal element $s(l')$ of $(r_h+L)\cap \mathcal{S}'$.
 \item[(2)] $s(l')$ can be found via the following
  computation sequence $\{z_i\}_i$ connecting $l'$ and $s(l')$:
   set $z_0:=l'$, and assume that $z_i$ ($i\geq 0$) is already constructed. If
   $(z_i,E_{v(i)})>0$ for some $v(i)\in\mathcal{V}$ then  set $z_{i+1}=z_i+E_{v(i)}$. Otherwise
   $z_i\in \mathcal{S}'$ and necessarily  $z_i=s(l')$.
\item[(3)] $\chi(z_{i+1})\leq \chi(z_i)$ along the (any such)  sequence.
\end{enumerate}
\end{lemma}
In general the choice of the individual vertex $v(i)$ might not be unique, nevertheless the final
output $s(l')$ is unique.

The proof is almost identical with the proof of Lemma \ref{lem:cs} and is left to the reader.

In particular, one can start with the cycle $l'=r_h$ and obtain  $s(r_h)=s_h$.
\begin{lemma}\label{lem:ZKs}  $Z_K\geq s_{[Z_K]}$.
\end{lemma}
\begin{proof}
If $\Gamma$ is a minimal resolution (that is, $b_0>1$) then by adjunction formulae
we get that $Z_K\in\mathcal{S}'$. Hence, by minimality of $s_{[Z_K]}$ one has $Z_K\geq s_{[Z_K]}$.

If $b_0=1$ then by Laufer criterion \cite{Laufer72} $(X,o)$ is non--rational, ie. $p_g>0$.
 Consider the computation
sequence $\{z_i\}_{i=0}^m$ from Lemma \ref{lem:cs2}{\it (2)} connecting $r_{[Z_K]}$ with $s_{[Z_K]}$.
We wish to show inductively that for each term $z_i$ one has $Z_K\geq z_i$ and
$h^1(\mathcal {O}_{Z_K-z_i})=p_g$.

For $i=0$ is true,  since  $Z_K\geq r_{[Z_K]}$ (cf. subsection \ref{sss:more}), and also
$h^1(\mathcal {O}_{Z_K-r_{[Z_K]}})=h^1(\mathcal {O}_{\lfloor Z_K\rfloor})=p_g$ (by
Kodaira or Grauert--Riemenschneider type vanishing
$h^1(\widetilde{X}, \mathcal{O}_{\widetilde{X}}(-\lfloor Z_K\rfloor))=0$).

Assume that for some $i$ one has $Z_K\geq z_i$, $h^1(\mathcal {O}_{Z_K-z_i})=p_g$,
 and $z_i<s_{[Z_K]}$.
Then the next term in the sequence is $z_{i+1}=z_i+E_{v(i)}$ with $(E_{v(i)}, z_i)>0$.

 Assume first that
the $E_{v(i)}$--coefficients of $Z_K-z_i$ is positive. Then  $Z_K\geq z_{i+1}$ holds.
Furthermore, in the cohomological long exact sequence of
$$ 0\to \mathcal{O}_{E_{v(i)}}(-Z_K+z_{i+1})\to \mathcal{O}_{Z_K-z_i}\to
\mathcal{O}_{Z_K-z_{i+1}}\to 0$$
the Chern number $(E_{v(i)}, -Z_K+z_{i+1})>-2$, hence
$h^1(\mathcal{O}_{E_{v(i)}}(-Z_K+z_{i+1}))=0$, which implies
$h^1(\mathcal{O}_{Z_K-z_{i+1}})=h^1(\mathcal{O}_{Z_K-z_i})=p_g$.

Assume next the opposite: $m_{v(i)}(Z_K-z_i)=0$.
Then $0<(E_{v(i)},z_i)= (E_{v(i)},z_i-Z_K)+(E_{v(i)},Z_K)\leq E^2_{v(i)}+2$, hence
necessarily $v(i)=v_0$ (since for other vertices $E^2_{v(i)}+2\leq 0$).
This means that $Z_K-z_i\in L_{\geq 0}$ is supported on the legs, but since these subgraphs are strings supporting rational singularities, this fact contradicts with $h^1(\mathcal {O}_{Z_K-z_i})=p_g>0$.
\end{proof}
\subsubsection{Duality properties}\label{ss:duality}

Consider the cycle $\{x(\ell)\}_{\ell\geq 0}$ and $\{x^* (\ell)\}_{\ell\geq 0}$ defined after Lemma
\ref{lem:cs}.
Let us define  $\delta:=\min \{\ell\,:\,  x^*(\ell)\in\mathcal{S}'\}$.

Then consider the concatenated sequences from Lemma \ref{lem:cs} connecting
$r_h\mapsto x^ 0\mapsto x^1\mapsto\cdots \mapsto x^\delta$ (if $h=0$ then this is the one--element
sequence $0$). This (with special choices of the vertices $v(i)$) satisfies
the requirements of the computation sequence from Lemma \ref{lem:cs2}.
This shows that
\begin{equation}\label{eq:xu2}
x^*(\delta)=s_{[Z_K]} \ \ \mbox{and} \ \  \delta=m_0(s_{[Z_K]}-r_{[Z_K]}).\end{equation}
Furthermore,  since  this
concatenated sequence has the property $\chi(z_{i+1})\leq \chi(z_i)$ by
Lemma \ref{lem:cs2}{\it (3)},
\begin{equation}\label{eq:xu}
(x^*(\ell),E_0)>0 \ \mbox{for all $\ell=0, \ldots, \delta-1$, however} \
(x^*(\delta), E_0)\leq 0.
\end{equation}
Set  $\Delta:= m_0(Z_K-r_{[Z_K]})$. By Lemma \ref{lem:ZKs} one has $\Delta\geq \delta$.

\begin{prop}\label{prop:chidual}  (a) The following `duality property' holds:
$$\chi(x^*(\ell))=\chi(x(\Delta-\ell)) \ \ \mbox{for $0\leq \ell\leq \Delta$}.$$
(b) Let $N^*(\ell)$ be defined similarly as $N(\ell)$, but for the $x^*$--sequence,
namely, $N^*(\ell):=-(x^*(\ell), E_0)$, cf. Lemma \ref{lem:cs}{\it (5)}. Then
$$N(\ell)+N^*(\Delta -1-\ell)=-2 \ \ \mbox{for any $0\leq \ell\leq \Delta-1$}.$$
(If $Z_K\in L$ then $x(\ell)=x^*(\ell)$, hence  $N(\ell)=N^*(\ell)$ too, and  we recover
Prop. \ref{prop:SYM}.)
\end{prop}
\begin{proof} (a) Note that
 $m_0(Z_K-x^*(\ell))=\Delta-\ell=m_0(x(\Delta-\ell))$.  Hence,
 by Lemma \ref{lem:cs}{\it (3)}  applied for $x(\Delta-\ell)$ gives
 $\chi(Z_K-x^*(\ell))\geq \chi(x(\Delta-\ell))$.
 Symmetrically, $m_0(x^*(\ell))=m_0(Z_K-x(\Delta-\ell))$, hence by the same argument applied for $x^*(\ell)$ gives
  $\chi(x^*(\ell))\leq \chi(Z_K-x(\Delta-\ell))$.  Finally observe that $\chi(Z_K-y)=\chi(y)$.

(b) By Remark \ref{rem:xt}, Lemma \ref{lem:cs}{\it (5)} and part (a) of this
Proposition one has
\begin{equation}\label{eq:dualN}
1+N(\ell)=\chi(x(\ell+1))-\chi(x(\ell))=\chi(x^*(\Delta-\ell-1))-\chi(x^*(\Delta-\ell))=
-1+(x^*(\Delta-\ell-1),E_0).\end{equation}
Then, by definition, the last term is  $-1-N^*(\Delta-\ell-1)$.
\end{proof}

\subsection{The Frobenius number of $\mathcal{M}$}

Let $M$ be a negative definite Seifert rational homology sphere and we consider the $\mathcal{S}$-module $\mathcal{M}$ associated with it. Recall that $p_g=P^+_0(1)$. Then   by (\ref{eq:polp}) and
(\ref{eq:defM}) we obtain that $\mathbb{Z}_{\geq 0}\subset \mathcal{M}$ if and only if   $p_g=0$,
cf. (\ref{rational}).

\begin{thm}\label{FrobM} Assume that $\Gamma$ is not
rational and set   $s:=m_0(s_{[Z_K]})$. Then
$$f_{\mathcal{M}}=\Delta-\delta-1=\gamma-s=m_0(Z_K-s_{[Z_K]})-1\geq 1.$$
In particular,   $\gamma=s+f_{\mathcal{M}}\geq s+1\geq 1$.
\end{thm}

\begin{proof} Consider again the identities from (\ref{eq:dualN}).
 The last term, by (\ref{eq:xu}),
 is $\geq 0$ for $\Delta-\ell-1\in \{0, \ldots, \delta-1\}$,
while it is $<0$ for $\Delta-\ell-1=\delta $.
This means that $\{\gamma-s+1, \dots, \gamma-m_0(r_h)\}\subset \mathcal{M}$, but
$\gamma-s\not\in\mathcal{M}$. Furthermore, $\ell\in\mathcal{M}$ for any $\ell>\gamma $ by
Proposition \ref{prop:INEQ}{\it (b)}.
\end{proof}

\begin{cor} Assume that $\Gamma$ is non--rational.

(a)
If  $\Gamma$ is numerically Gorenstein, then $s_{[Z_K]}=0$, 
hence  $f_{\mathcal{M}}=\gamma$. This is compatible with Corollary \ref{cor:frob}{\it (a)} and
with the integral  homology sphere case from  \cite{CK}.

(b) Assume that $\gamma\in \mathbb{Z}$, but $Z_K\not\in L$ (see eg.
$Sf=(-1,0; (6,1)_{i=1}^4)$).   Then
$s_{[Z_k]}\geq r_{[Z_K]}\not=0$. But $s_{[Z_k]}$ being a non--zero element of $\mathcal{S}'$,
all its $E_v$--coefficients are strict positive. Hence $s>0$ and $f_{\mathcal{M}}<\gamma$.

(c) If $(X,o)$ is weighted homogeneous normal surface singularity then
$\deg P^+_0=\gamma-s$.
\end{cor}

\begin{remark}
The authors know no closed formula for the number $s$ in terms of the Seifert invariants.
By the definition it is the $E_0$--coefficient of $s_{[Z_K]}$ (hence, it can be determined
from the  $E^*_v$-coefficients of $s_{[Z_K]}$ too). However, the only fact known
about  $s_{[Z_K]}$ is  that the
 $E^*_v$-coefficients are determined as  a solution of
  a system of Diophantine quasipolynomial inequalities (see \cite[Proposition 11.5]{NOSZ}).
  Therefore, the formula for the Forbenius number given by the Theorem \ref{FrobM} is explicit up to the number $s$ given by a Laufer computation sequence associated with the class $[Z_K]$
  (or the solution of that system of inequalities).
\end{remark}

\begin{example}\label{ex}
Let us denote by $M_{(70)}$ the Seifert rational homology sphere associated with the following plumbing graph:

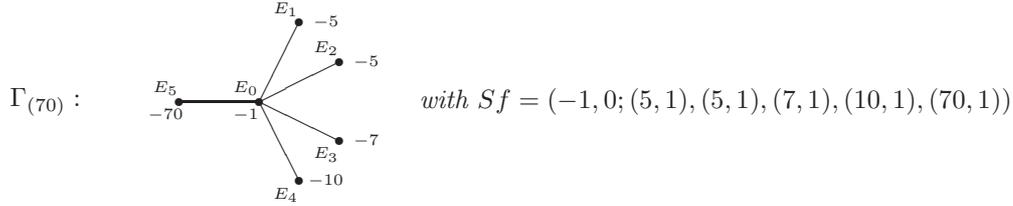
\begin{picture}(200,80)(-50,-10)
\put(30,30){\circle*{3}}
\put(60,30){\circle*{3}}
\put(90,45){\circle*{3}}\put(90,15){\circle*{3}}
\put(75,60){\circle*{3}}\put(75,0){\circle*{3}}
\put(30,30){\line(1,0){30}}
\put(60,30){\line(2,1){30}}\put(60,30){\line(2,-1){30}}
\put(60,30){\line(1,2){15}}\put(60,30){\line(1,-2){15}}

\put(25,35){\makebox(0,0){\tiny{$E_5$}}}
\put(25,25){\makebox(0,0){\tiny{$-70$}}}

\put(55,35){\makebox(0,0){\tiny{$E_0$}}}
\put(55,25){\makebox(0,0){\tiny{$-1$}}}

\put(85,50){\makebox(0,0){\tiny{$E_2$}}}
\put(100,45){\makebox(0,0){\tiny{$-5$}}}

\put(85,10){\makebox(0,0){\tiny{$E_3$}}}
\put(100,15){\makebox(0,0){\tiny{$-7$}}}

\put(70,65){\makebox(0,0){\tiny{$E_1$}}}
\put(85,60){\makebox(0,0){\tiny{$-5$}}}

\put(70,-5){\makebox(0,0){\tiny{$E_4$}}}
\put(85,0){\makebox(0,0){\tiny{$-10$}}}

\put(-20,30){\makebox(0,0){$\Gamma_{(70)}:$}}
\put(120,30){\makebox(0,0)[l]{$\mbox{with} \ Sf=(-1,0; (5,1), (5,1), (7,1), (10,1), (70,1))$}}
\end{picture}

We write  all the cycles $l'=\sum_{i=0}^5 l'_i E_i$ in the form $l'=(l'_0,l'_1,\dots,l'_5)$.
By solving the adjunction equations (\ref{eq:adjun}) we get $Z_K=(47/6,13/6,13/6,11/6,19/12,13/12)$,
 hence $\gamma=41/6$ and
$\Gamma_{(70)}$ is not numerically Gorenstein.
Clearly $r_{[Z_K]}=(5/6,1/6,1/6,5/6,7/12,1/12)$.   $s_{[Z_K]}$ is calculated by the Laufer computation sequence from section \ref{ss:concat} and we get
%as follows: since $(z^0=r_{[Z_K]},E_0)=1$ we continue the sequence with $z^1_0=r_{[Z_K]}+E_0$ for
%which $(z^1_0,E_i)>0$ for $i=1,2$, hence one can choose eg. $z^1_1:=r_{[Z_K]}+E_0+E_1$
%and $z^1_2:=r_{[Z_K]}+E_0+E_1+E_2$. The last two step is provided by $(z^1_2,E_0)=2$, hence
$$s_{[Z_K]}=r_{[Z_K]}+3 E_0+E_1+E_2=(23/6,7/6,7/6,5/6,7/12,1/12).$$
By Theorem \ref{FrobM} the  Frobenius number of the module $\mathcal{M}_{(70)}$ is $\gamma-s=41/6-23/6=3$.
\end{example}

\section{The semigroup $\mathcal{S}_M$}

\subsection{Semigroups generated by periodic subadditive functions}
We consider the function $f_M:\mathbb{Z}_{\geq 0}\to \mathbb{Z}_{\geq 0}$, $f_M(\ell)=\frac{1}{|e|} \sum_{i=1}^d\{-\omega_i\ell/\alpha_i\}$, where we define the rational part as ${x}:=x-[x]$ for any $x\in\mathbb{Q}$. Then it is a subadditive function with period
 $\alpha={\rm lcm}(\alpha_1,\dots,\alpha_d)$ in the following sense:
 $$ f(\ell_1+\ell_2)\leq f(\ell_1)+f(\ell_2), \ \ f(0)=0,  \ \ \ \mbox{and} \ \ \ f(\ell+\alpha)=f(\ell) \ \ \ \mbox{for every} \ \ \ell_1,\ell_2,\ell\in\mathbb{Z}_{\geq 0}.$$
Usually,  such a function defines a semigroup
$\mathcal{S}_f=\{\ell\in\mathbb{Z}_{\geq 0}\ \, : \, f(\ell)\leq \ell\}$,  see \cite{Rosales}.
 In this case, for the above function $f_M$, the semigroup associated with $f_M$
  is exactly the numerical semigroup $\mathcal{S}_M$.

Again, by \cite{Rosales},
the Ap\'ery set of $\mathcal{S}_f$ with respect to an element $\alpha\in\mathcal{S}$ consists of $\lceil (f(\ell)-\ell)/\alpha \rceil \alpha +\ell$ for $\ell=0,\dots,\alpha-1$. Thus,  in our case,
(where $\alpha ={\rm gcd}\{\alpha_i\}$ as  in the previous sections)
$$Ap(\mathcal{S}_M),\alpha)=\{\lceil -N(\ell)/\mathfrak{o} \rceil\alpha+\ell\, : \, \ell=0,\dots, \alpha-1\}.$$
 Then  the Frobenius number can be expressed by Selmer's formula as
 $$f_{\mathcal{S}_M}=\max \{Ap(\mathcal{S}_{M},\alpha)\}-\alpha=
 \textstyle{\max_{\ell=0}^{\alpha-1}}\, \{\lceil -N(\ell)/\mathfrak{o} \rceil\alpha+\ell\}-\alpha.$$
Furthermore, the number of gaps $\#\{\mathbb{N}\setminus \mathcal{S}_M \}$ is
$$\frac{1}{\alpha}\sum _{w\in Ap(\mathcal{S}_{M},\alpha)} w- \frac{\alpha-1}{2}=
\sum _{\ell=0}^{\alpha-1} \Big\lceil -\frac{N(\ell)}{\mathfrak{o}}\Big\rceil.$$
 Though the above formula
 gives a method to compute $f_{\mathcal{S}_M}$, still is difficult to use it
  to get an explicit closed--form expression in terms of Seifert invariants.

 The goal of this section is to find a different expression for
 $f_{\mathcal{S}_M}$ in terms of the  lattice $L$.

\subsection{The Frobenius number of $\mathcal{S}$} \
The idea is
to construct a new manifold out of the original one,
 and to prove that the Frobenius number of the module associated with the new manifold
 coincides with the Frobenius number of the semigroup associated with the original Seifert rational homology sphere. Since the Frobenius number of modules is determined in Theorem \ref{FrobM},
 we get the desired $f_{\mathcal{S}_M}$ as well.

\subsubsection{}
We consider a negative definite Seifert rational homology sphere $M$ with Seifert invariants $(-b_0;(\alpha_i,\omega_i)_{i=1}^d)$ and its associated semigroup $\mathcal{S}_{M}$ defined by the quasi-linear function $N(\ell)$.

Let $M_{(n)}$ be a newly defined Seifert rational homology sphere characterized by the Seifert invariants $(-b_0;(\alpha_i,\omega_i)_{i=1}^d),(n,1))$ for some integer $n\geq 1/|e|$. By this choice of $n$ the 
new graph is also negative definite since its orbifold Euler number $e_{(n)}=e+1/n<0$. We can associate with $M_{(n)}$ the module $\mathcal{M}_{(n)}$ defined by equation (\ref{eq:defM}). Then we have the following comparison result.

\begin{prop}\label{prop:comp}
For big enough $n\gg0$, $\mathcal{M}_{(n)}$ becomes a semigroup independent of $n$,
 and it coincides with $\mathcal{S}_M$.
\end{prop}
\begin{proof}
Let $N_{(n)}(\ell)$ be the quasi-linear function associated with $\mathcal{M}_{(n)}$. Then $N(\ell)=N_{(n)}(\ell)+\lceil\ell/n\rceil$ ($\dag$). Therefore, for $\ell>0$,
 $\ell\in \mathcal{M}_{(n)}$ implies $\ell \in \mathcal{S}_{M}$.

 For the opposite inclusion we need some preparation. First we observe that
 \begin{equation}
N(\ell)\geq t\cdot \mathfrak{o}-1 \ \ \ \mbox{for any} \ \ \ell> \gamma-s+t\cdot \alpha \ \ \mbox{and} \ \ t\in\mathbb{Z}_{\geq 0},
\end{equation}
 Indeed, one has $N(\alpha)=\mathfrak{o}$, hence the $t=0$ case is true since $\ell>\gamma-s$ implies $N(\ell)\geq -1$ by Theorem \ref{FrobM}. Then we can proceed by induction using
 $N(\ell+\alpha)=N(\ell)+\mathfrak{o}$, cf. Proposition \ref{prop:INEQ}{\it (e)}.

This system implies  that for $\gamma-s+t\cdot \alpha< \ell\leq \gamma-s+(t+1)\cdot \alpha$
($\ddag$)
one has $1+N_{(n)}(\ell)=1+N(\ell)-\lceil\ell/n\rceil
\geq t\cdot \mathfrak{o} -\lceil\ell/n\rceil$.
Then we choose $n$ sufficiently large such that $t\cdot \mathfrak{o} -\lceil\ell/n\rceil$
is non--negative for any $t\geq 1$ and for any $\ell$ satisfying ($\ddag$), and additionally 
$n\geq \gamma-s+\alpha$.

Now we can conclude. Take
$\ell \in \mathcal{S}_{M}$. If $0<\ell\leq \gamma-s+\alpha$ then $\ell\leq n$ too,
hence by the expression
($\dag$) one has $\ell\in \mathcal{M}_{(n)}$. Otherwise,
$1+N_{(n)}(\ell)
\geq t\cdot \mathfrak{o} -\lceil\ell/n\rceil\geq 0$.
 Hence $\ell\in \mathcal{M}_{(n)}$ again.
\end{proof}

\subsubsection{The $b_0\geq d$ case}\label{bnagy} 
Since $N(1)=b_0-\sum_i\lceil \omega_i/\alpha_i\rceil=b_0-d$ we obtain that $\mathcal{S}_M=
\mathbb{N} \, \Leftrightarrow\, N(1)\geq 0\ \Leftrightarrow\ b_0\geq d$. 

Maybe is worth to mention that if  $b_0\geq d$  then the fundamental cycle of $\Gamma$ 
(the minimal element of $\mathcal{S}\setminus \{0\}$) is reduced , hence $(X,o)$ is rational 
by \cite{Laufer72} (singularities with reduced fundamental cycle are called `minimal rational').

In fact, in such cases, the graph of $M_{(n)}$ is also rational (for  a proof see \cite[Th. 4.1.3]{NG}).
Hence we can also argue  as  follows: $\mathcal{M}_{(n)}=\mathbb{N}$ by (\ref{rational}) and 
$\mathcal{S}_{M}=\mathcal{M}_{(n)}$ by Proposition \ref{prop:comp}. 

\subsubsection{The $b_0< d$ case}\label{bkicsi} By the above discussion if $b_0<d$ then 
 $f_{\mathcal{S}_M}$ is well--defined and is $\geq 1$. 

We denote by $\check{s}$ the $E_0$-coefficient of $s_{[Z_K+E_0^*]}$,  where (via the standard notation)
$s_{[Z_K+E_0^*]}$ is the unique minimal element in $\mathcal{S}'$ associated with $[Z_K+E_0^*]\in H$.

\begin{thm}\label{FrobS} If $b_0<d$ then 
$$f_{\mathcal{S}_M}=\gamma+\frac{1}{|e|}-\check{s}.$$
\end{thm}

\begin{example}\label{ex:4db} \
(1) If $Z_K\in L$ and $\mathfrak{o}=1$ (eg. if $M$ is an integral homology sphere) then $1/|e|=\alpha$,
$\check{s}=0$, hence $f_{\mathcal{S}_M}=\gamma+\alpha$, compare with  Section \ref{s:Ssf}.

(2) If $\mathfrak{o}=1$ then $[E^*_0]=0$, hence $\check{s}=s$. Therefore
$f_{\mathcal{S}_M}=\gamma+\alpha-s$.

(3) If $Z_K\in L$  then  $1/|e|=\alpha/\mathfrak{o}$. (Note also that always is true
that $m_0(E^*_0)=-(E^*_0,E^*_0)=\prod_i\alpha_i/|H|=1/|e|$.)
Since $\check{s}=m_0(s_{[E^*_0]})$ and $E^*_0$ is already in $\mathcal{S}'$,
we get $E^*_0-s_{[E^*_0]}\in L_{\geq 0}$, hence $1/|e|\geq \check{s}$. In particular,
$f_{\mathcal{S}_M}=\gamma+ m_0(E^*_0-s_{[E^*_0]})\geq \gamma$.

This combined with 
 Corollary \ref{cor:SMsym} gives 
 $\min\{\mathcal{M}_M\}=\gamma-f_{\mathcal{S}_M}=\check{s}-1/|e|=
 -m_0(E^*_0-s_{[E^*_0]})$. 

(4) In general, $s_{[Z_K+E_0^*]}=s_{[s_{[Z_K]}+E_0^*]}$, but $s_{[Z_K]}+E_0^*\in \mathcal{S}'$, hence
$s_{[Z_K]}+E_0^*- s_{[Z_K+E_0^*]}\in L_{\geq 0}$, which implies
$s+1/|e|\geq \check{s}$. That is, $f_{\mathcal{S}_M}\geq \gamma-s$. Note that
$\gamma-s=f_{\mathcal{M}_M}$. This shows that
$f_{\mathcal{S}_M}-f_{\mathcal{M}_M}=m_0(s_{[Z_K]}+E_0^*- s_{[Z_K+E_0^*]})$.
\end{example}

\begin{example}
Before we proceed to the proofs, let us ilustrate the previous statements by further discussion on Example \ref{ex}.

Let $M$ be the Seifert rational homology sphere with $Sf=(-1,0;(5,1),(5,1),(7,1),(10,1))$. The claim is that $n=70$ is big enough in order to have Proposition \ref{prop:comp} to be valid. Thus, the module $\mathcal{M}_{(70)}$ of the manifold $M_{(70)}$ is a semigroup, same as $\mathcal{S}_{M}$, and the calculations from Example \ref{ex} gives $f_{\mathcal{S}_{M}}=3$. We are going to present the calculation using the formula of Theorem \ref{FrobS} too.

Associated with $M$,  the adjunction formulae (\ref{eq:adjun}) give $Z_K=(24/5,39/25,39/25,7/5,32/25)$, where $l'=(l'_0,l'_1,\dots,l'_4)$ represents the cycle $l'=\sum_{i=0}^4 l'_i E_i$ (see the graph from Example \ref{ex}). Hence, $\gamma=19/5$ and one can also calculate $e=-5/14$. On the other hand, by the calculation of $E^*_0=(14/5,14/25,14/25,2/5,7/25)$ we deduce that $r_{[Z_K+E^*_0]}=(3/5,3/25,3/25,4/5,14/25)$. Then by running a general Laufer computation sequence (cf.
Lemma  \ref{lem:cs2}) we find that $s_{[Z_K+E^*_0]}=r_{[Z_K+E^*_0]}+3E_0+E_1+E_2$, hence $\check{s}=18/5$. Finally, we get $f_{\mathcal{S}_M}=19/5+14/5-18/5=3$.

\end{example}

\subsection{The proof of Theorem \ref{FrobS}}\label{ss:jj*}
We denote by $\Gamma_{(n)}$ the plumbing graph associated with $M_{(n)}$ and let $L_{(n)}$ and $L'_{(n)}$ be the associated lattices as in \ref{ss:nsstop}. One can consider the (homological)
inclusion operator $j_{(n)}:L\to L_{(n)}$ identifying naturally the corresponding $E$-base elements
 indexed by $\mathcal{V}(\Gamma)$ in the two lattices. This preserves the intersection forms. Moreover, it extends naturally to the rational cycles as well, in particular to
$j_{(n)}:L'\to L'_{(n)}$.
Additionally, denote by $E_{+}$ the new base element of $L_{(n)}$ associated with the newly created vertex $v_+$ of $\Gamma_{(n)}$.

Set the notation $E^*_{v,(n)}$ for the dual base elements in $L'_{(n)}$ associated with $j_{(n)}(E_v)$.
Let $j^*_{(n)}:L'_{(n)}\to L'$ be the (cohomological) dual operator of $j_{(n)}$, defined by $j^*_{(n)}(E^*_{v,(n)})=E^*_{v}$ for $v\in\mathcal{V}(\Gamma)$ and $j^*_{(n)}(E^*_{+})=0$.
Then for every $l'\in L'_{(n)}$ and $l\in L$ one has the projection formula
\begin{equation}\label{eq:proj}
(j^*_{(n)}(l'),l)_{\Gamma}=(l',j_{(n)}(l))_{\Gamma_{(n)}}.
\end{equation}
This imply that  $j^*_{(n)}(E_+)=-E^*_0\in L'$ (usually  $\not\in L$). Hence $j^*_{(n)}(E_++j_{(n)}E^*_0)=0$. One also  sees that $j^*_{(n)}(j_{(n)}E_v)=E_v$.

Set $H_{(n)}:=L'_{(n)}/L_{(n)}=H_1(M_{(n)},\mathbb{Z})$.
Denote by $Z_{K(n)}$ the anti-canonical cycle in $L'_{(n)}$ and consider also
the cycle $s_{[Z_{K(n)}]}$ associated with the class $[Z_{K(n)}]\in H_{(n)}$.
%Finally  take $j^*_{(n)}(s_{[Z_{K(n)}]})$.

\begin{lemma}[\cite{LSzPoin}, \cite{LNN}]\label{lem:1}
For any $h\in H_{(n)}$ one has $j^*_{(n)}(s_{h})=s_{[j^*_{(n)}(s_{h})]}\in L'$.
In particular, $j^*_{(n)}(s_{Z_{K(n)}})=s_{[j^*_{(n)}(s_{Z_{K(n)}})]}$.
\end{lemma}
By a computation (based on adjunction formulae) one verifies  that 
\begin{equation}\label{eq:ZKn}
Z_{K(n)}=j_{(n)}Z_K+c_{(n)}\cdot (E_++j_{(n)}E^*_0), \ \ 
\mbox{where} \ \ 
c_{(n)}=\frac{n+\gamma-1}{n-1/|e|}. \end{equation}
The inequality 
$c_{(n)}<1$ is equivalent with $|e|\gamma<|e|-1$, or $d+\sum_i (\omega_i-1)/\alpha_i <b_0+1$.
This last inequality implies $b_0\geq d$. Hence, by our assumption $b_0<d$ we automatically have 
$c_{(n)}\geq 1$. Since $c_{(n)}<2$ for $n\gg 0$, we get that 
\begin{equation}\label{eq:rhproof} 
r_{Z_{K(n)}}- (c_{(n)}-1)E_+\in  j_{(n)}(L\otimes \mathbb{Q}).
\end{equation}
Analysing the algorithm from  Lemma \ref{lem:cs2} with  $n\gg0$
we obtain that  in the computation sequence 
connecting $r_{Z_{K(n)}}$ with $s_{Z_{K(n)}}$ we never add $E_+$, hence (\ref{eq:rhproof}) is valid
for $s_{Z_{K(n)}}$ too. In particular,  there exists $l_{(n)}\in L\otimes \mathbb{Q}$ such that 
\begin{equation}\label{eq:shproof} 
s_{Z_{K(n)}}= (c_{(n)}-1)E_++j_{(n)}(l_{(n)}). 
\end{equation}
Since $Z_{K_{(n)}}-E_+-r_{Z_{K(n)}}\in j_{(n)}(L)$, by the above statement 
 $Z_{K_{(n)}}-E_+-s_{Z_{K(n)}}\in j_{(n)}(L)$ too, hence 
 $[j_{(n)}^*s_{Z_{K(n)}}] =[j_{(n)}^*(Z_{K(n)}-E^+)]$. This last term via (\ref{eq:ZKn}) 
 (and the properties of $j_{(n)}^*$ listed above) is $[Z_K+E^*_0]$. Hence Lemma \ref{lem:1} reads as
$j_{(n)}^* s_{Z_{K(n)}}= s_{[Z_K+E^*_0]}$. This applied to (\ref{eq:shproof}) gives 
 \begin{equation}\label{eq:jproof} 
 s_{[Z_K+E^*_0]}=j^*_{(n)}((c_{(n)}-1)E_++j_{(n)}(l_{(n)}))=-(c_{(n)}-1)E^*_0+ l_{(n)}.
\end{equation}
 We wish to compute $s_{(n)}$,  the $j_{(n)}E_0$--coefficient of $s_{Z_{K(n)}}$. By 
 (\ref{eq:shproof}) it is the $E_0$--coefficient of $l_{(n)}$, which by (\ref{eq:jproof})
is $\check{s}+(c_{(n)}-1)/|e|$. 

On the other hand, by a computation, the $\gamma$--invariant of $M_{(n)}$ satisfies 
$\gamma_{(n)}=\gamma+c_{(n)}/|e|$. 

Then, by Theorem \ref{FrobM} one has $f_{\mathcal{M}_{(n)}}=\gamma_{(n)}-s_{(n)}=
\gamma+c_{(n)}/|e|-(\check{s}+(c_{(n)}-1)/|e|)=\gamma+1/|e|-\check{s}$. Then apply Proposition 
\ref{prop:comp}.

\section{Is $\mathcal{S}_M$ symmetric?}

\subsection{} In the integral homology sphere case we have seen (cf. Section \ref{s:Ssf})
that $\mathcal{S}_M$ is symmetric: $\ell\in \mathcal{S}_M\ \Leftrightarrow f_{\mathcal{S}_M}-\ell
\not\in \mathcal{S}_M$. It is natural to ask whether  this fact extends to the case of 
rational homology spheres or not. 
The next example shows that the answer in general is no.

\begin{example}\label{ex:NO1}
Take the Seifert 3-manifold given by $(-1,0; (4,1),(4,1), (4,1), (10,1), (40, 1))$.
Then $Z_K=(18, 5,5,5, 13/5, 7/5)$ and $E_0^*=(8,2,2,2,4/5, 1/5)$.
Hence $r_{[Z_K+E^*_0]}=(0,0,0,0,2/5,3/5)$. By the algorithm from
Lemma \ref{lem:cs2} $s_{[Z_K+E^*_0]}=(4,1,1,1,2/5,3/5)$, thus $\check{s}=4$.
 Furthermore,  $1/|e|=8$ and  $\gamma=17$.
Therefore, by Theorem \ref{FrobS},
$f_{\mathcal{S}_M}=17+8-4=21$. 

On the other hand, by a computation, $N(4)=N(7)=N(10)=N(11)=N(14)=N(17)=-1$, 
 hence the elements 4,7,10,11, 14, 17 do not belong to $\mathcal{S}_M$ but 
(4,17), (7,14), (10,11) are symmetric with respect to $f_{\mathcal{S}_M}$. 
\end{example}

Usually symmetries in algebraic geometry (and algebra) are induced by some Gorenstein property.
Hence, we can try to restrict ourselves to the numerically Gorenstein topological type, and ask the same 
question: is $\mathcal{S}_M$ symmetric?  In this case $f_{\mathcal{S}_M}$ is computed in 
Example \ref{ex:4db}{\it (c)}, $\min\{\mathcal{M}_M\}+f_{\mathcal{S}_M}=\gamma$
by Corollary \ref{cor:SMsym}, and 
%\begin{equation}\label{SHSsym22}
$N(\min\{\mathcal{M}_M\}+\ell)+N(f_{\mathcal{S}_M}-\ell)=-2$ for any $ \ell\in\mathbb{Z}$
by (\ref{SHSsym}). 
%\end{equation}
\begin{lemma}
$\mathcal{M}_M=\min\{\mathcal{M}_M\}+\mathcal{S}_M$ if and only if $\mathcal{S}_M$ is symmetric.
\end{lemma}
\begin{proof}
Consider the following equivalences:
$$ \ell\in\mathcal{S}_M\ \Leftrightarrow\ \min\{\mathcal{M}_M\}
+\ell\in \min\{\mathcal{M}_M\}+\mathcal{S}_M$$
$$f_{\mathcal{S}_M}-\ell\not\in \mathcal{S}_M\ \Leftrightarrow \
N(f_{\mathcal{S}_M}-\ell)+1\leq 0\ \Leftrightarrow \ N(\min\{\mathcal{M}_M\}+\ell)+1\geq 0\
\Leftrightarrow \ \min\{\mathcal{M}_M\}+\ell\in \mathcal{M}_M.
$$
Then compare the first and last terms. 
\end{proof}
For numerically Gorenstein graphs the symmetry $N(\ell)+N(\gamma-\ell)=-2$ can be interpreted as 
a consequence of the Gorenstein property of weighted homogeneous singularities,
however, it is not clear at all if this fact implies 
 any of the properties from the above Lemma.

The point is that $\mathcal{S}_M$ in general is not symmetric even if we restrict 
ourselves to the numerically Gorenstein case. 

\begin{example}\label{ex:NO2}
Consider the  Seifert 3-manifold given by $(-2,0; (2,1), (2,1), (3,1), (3,1), (7,1), (7,1), (84,1))$.
Then $Z_K=(86, 43, 43, 29, 29, 13, 13, 2)$, hence $\Gamma$ is numerically Gorenstein.
By a computation $1/|e|=28$, and $\gamma=85$. Furthermore, 
$E_0^*=s_{[E^*_0]}=(28, 14, 14, 28/3, 28/3, 4, 4, 1/3)$ hence $\check{s}=1/|e|=28$. 
In particular, $f_{\mathcal{S}_M}=\gamma=85$.

 On the other hand, $N(6)=N(85-6)=-1$, hence 6 and 85-6 are symmetric but none of them belong to
 $\mathcal{S}_M$. (The same is true for the pairs (12, 85-12), (14, 85-14) and (18, 85-18). 

(In fact, in this case $\min\{\mathcal{M}_M\}=0$, but $\mathcal{M}_M \not= \mathcal{S}_M$. Moreover, 
the graph can also be interpreted as $\Gamma_{(84)}$ associated with its subgraph, 
whose consequences can be tested by the interested reader.)

\end{example}

\section{Problem: classification of $\mathcal{S}_M$ for Seifert rational homology sphere}

\subsection{}
We start this section by the following question  regarding the semigroups $\mathcal{S}_M$ associated with Seifert rational homology sphere. 

\begin{question}
How can the numerical semigroups $\mathcal{S}_M$ be classified based on the peculiar geometry/topology of Seifert rational homology spheres?
\end{question}
As we have already shown is section \ref{s:Ssf}  
the above question can be completely answered in the case of Seifert homology spheres. Indeed, in that case,   $\mathcal{S}_M$ is completely characterized by
its minimal set of generators, which are exactly the Seifert invariants of the manifold.
However, in general,  for nontrivial $H$, the characterization is more  difficult:
for the description of $\mathcal{S}_{an}$ we need to understand the structure (generators) of the  
integral monomial cycles. 

This leads to the following  more precise task.

\begin{question}
Determine the minimal set of generators of $\mathcal{S}_M$ in terms of the Seifert invariants!
\end{question}

In the sequel we clarify this problem completely in the case of 
Brieskorn--Hamm rational homology spheres.

\subsection{\bf Brieskorn--Hamm rational homology spheres}
We consider positive integers $a_i\geq 2$ ($i=1,\dots,n$) for some fixed $n\geq 3$. For a generic set of complex numbers $c=\{c_{ji}\}_{i=1,\dots,n}^{j=1,\dots,n-2}$ we consider the Brieskorn--Hamm isolated complete intersection singularity $(X_c(a_1,\dots,a_n),0)$ where
$$X_c(a_1,\dots,a_n):=\{z\in \mathbb{C}^{n} \ | \ c_{j1}z_1^{a_1}+\dots+c_{jn}z_n^{a_n}=0 \ \ \mbox{for every } \ 1\leq j\leq n-2\}.$$
Then the link $M=\Sigma(a_1,\dots,a_n)$ of $(X_c(a_1,\dots,a_n),0)$ is independent of the choice of
$c$ \cite{hamm}. It is 
an oriented Seifert 3-manifold with Seifert invariants
$$Sf=(-b_0,g;\substack{\underbrace{(\alpha_1,\omega_1),\dots,(\alpha_1,\omega_1)}\\ \tiny{s_1}},\dots ,\substack{\underbrace{(\alpha_n,\omega_n),\dots,(\alpha_n,\omega_n)}\\ \tiny{s_n}}),$$
where all the entries are determined explicitly from the integers $\{a_i\}_i$, for details see
\cite{JN,NeR}. 

Furthermore,  by \cite[Proposition 6.3]{NNsw} one has a complete characterization of the cases when
 $M$ is a rational homology sphere. Namely, $g=0$ if and only if the set $(a_1,\dots,a_n)$ (after a possible permutation) has exactly one of the following forms:
\begin{itemize}
 \item[(i)] $(a_1,\dots,a_n)=(m\cdot p_1,m\cdot p_2,p_3,\dots,p_n)$, where the integers $\{p_i\}_{i=1}^n$ are pairwise relative prime, and $\gcd(m,p_i)=1$ for any $i\geq 3$.
     In this case $\alpha_i=p_i$ for all $i$. 
     Moreover, $s_1=s_2=1$ and $s_i=m$ for any $i\geq 3$. If we set $q_i:=\mathrm{lcm}(a_1,\dots,a_n)/a_i$ then $\omega_i$ is determined from the congruence $\omega_iq_i\equiv-1 \ ({\rm mod} \
      \alpha_i)$, while $b_0$ is expressed from
 \begin{equation}\label{eq:BH1e}
 \mathfrak{o}=|e|\cdot\alpha_1\dots \alpha_n=1
 \end{equation}
 using the fact that $[E^*_0]$ is trivial in $H$.
 \item[(ii)] $(a_1,\dots,a_n)=(2^c\cdot p_1,2 p_2,2p_3,p_4,\dots,p_n)$ where $p_i$ 
  are odd and pairwise relative prime, and $c\geq 1$. In this case $\alpha_1=2^{c-1}p_1$
  and $\alpha_i=p_i$ for $i\geq 2$. Furthermore,
  $s_i=2$ for $i\leq 3$ and $s_i=4$ for $i\geq 4$. One determines  $\omega_i$ similarly as in (i),
 using the fact that in this case $\mathfrak{o}=2$.
\end{itemize}

\begin{thm}
The minimal set of generators for the semigroup $\mathcal{S}_M$ of Brieskorn--Hamm rational homology spheres corresponding to the above cases  are
\begin{itemize}
 \item[(i)] $\alpha_2\dots\alpha_n,\  \alpha_1\alpha_3\dots\alpha_n$ \ and \ $m\cdot\alpha_1\dots\alpha_{i-1}\alpha_{i+1}\dots\alpha_n$ \ for every $3\leq i\leq  n$;
 \item[(ii)] $\alpha_1\dots\alpha_{i-1}\alpha_{i+1}\dots\alpha_n$ for $i\in\{1,2,3\}$ \ and \ $2\cdot\alpha_1\dots\alpha_{i-1}\alpha_{i+1}\dots\alpha_n$ for $4\leq i\leq  n$.
\end{itemize}
\end{thm}

\begin{proof}
We extend the proof of Theorem \ref{thm:Sgen}. The fact that the elements of the semigroup 
$\mathcal{S}_M$ are the $E_0$--coefficients of the principal cycles from $\mathcal{S}_{an}$ 
is valid also for any weighted homogeneous singularity rational homology sphere link. 
Thus, we have $\mathcal{S}_M=\{(l,-E^*_0) : l\in\mathcal{S}_{an}\}$. On the other hand, 
by  the combinatorial characterization \cite[Theorem 7.1.2]{Nsplice} (see Theorem \ref{th:San} here) 
$l\in\mathcal{S}_{an}$ if and only if $l=\min_k \ \sum_{v\in\mathcal{E}} n_e^{(k)}E_e^*$ for a finite set $\{\sum_{v\in\mathcal{E}} n_e^{(k)}E_e^*\}_k \subset L$,
 where $n_e^{(k)}\in\mathbb{Z}_{\geq 0}$ and $\mathcal{E}$ is the set of end-vertices of the graph $\Gamma$. Thus, first we have to characterize the set of integers $\{n_e\}_{e\in\mathcal{E}}$ 
which provide an integral  cycle $\sum_{v\in\mathcal{E}} n_eE_e^*\in L$.

We discuss case (i). Let us denote by $E^*_1, E^*_2$ and $E^*_{ij}$ ($3\leq i\leq n$ and 
$1\leq j\leq m$) the duals of the end--vertices. In fact, the classes $[E^*_0]$, $[E^*_1]$ and $[E^*_2]$ are trivial in $H$, moreover one has (cf. \cite[pg. 301]{NNsw})
\begin{equation*}
H\simeq \oplus_{i\geq 3}\big\langle [E^*_{ij}], 1\leq j\leq m \ : \ \alpha_i[E^*_{ij}] = 0 \ \ \mbox{for all} \ j, \ \mbox{and} \ \sum_j[E^*_{ij}]=0\big\rangle\simeq \oplus_{i\geq 3}(\mathbb{Z}_{\alpha_i})^{m-1}.
\end{equation*}
This implies that the generating monomial cycle  are 
$E_1^*$, $E_2^*$, and for $i\geq 3$ cycle of type 
$\sum_{j=1}^m n_{ij}E^*_{ij}\in L$ ($n_{ij}\in \mathbb{Z}_{\geq 0}$). 
This last condition says that $\sum_{j=1}^m n_{ij}[E^*_{ij}]=0$ in $H$, hence by the above description of $H$,  $\sum_{j=1}^m n_{ij}E^*_{ij}\in L$ (with all $n_{ij}\in\mathbb{Z}_{\geq 0}$)
if and only if 
\begin{equation}\label{eq:nij}
(n_{ij})_j\in\mathbb{Z}\langle (\alpha_i,0, \ldots,0), \ldots, (0,\ldots, 0, \alpha_i),
(1,1,\ldots, 1)\rangle, \ \ \ \ (n_{ij}\in\mathbb{Z}_{\geq 0}).\end{equation}
Hence, there exists integers $k_0,k_1,\ldots, k_m$ such that $n_{ij}=k_j\alpha_i+k_0$.
This implies that all the $\alpha_i$--remainders are the same, say $k_0'$, and, in fact,
 $n_{ij}=k_j'\alpha_i+k_0'$ for non--negative integers  $k_0',k_1',\ldots, k_m'$. 

This implies that $\sum_jn_{ij}=A\alpha _i+Bm$ for non--negative integers $A$ 
and $B$, and the pair $A=0$ and $B=1$  can be  realized.

Write $\hat{\alpha_i} := \alpha_1\dots\alpha_{i-1}\alpha_{i+1}\dots\alpha_n$. Then 
using the identity (\ref{eq:BH1e}) and formulae and the standard formulae for $-(E^*_v, E^*_u)$ we get 
\begin{equation*}
(E_i^*,-E_0^*)=\hat{\alpha_i} \ (i=1,2), \ \ (E_{ij}^*,-E_0^*)=\hat{\alpha_i} \ (i\geq 3). 
\end{equation*}
Hence,  the generators of $\mathcal{S}_M$ are $(E_i^*,-E_0^*)=\hat{\alpha_i}$ for $i=1,2$, and 
$(\sum_jn_{ij}E_{ij}^*,-E_0^*)$ whenever $i\geq 3$ and  $(n_{ij})_j$ satisfies (\ref{eq:nij}).
This last term has the form  $(A\alpha_i+Bm)\hat{\alpha_i}=A\prod_i\alpha_i+Bm\hat{\alpha_i}$.
Since $\prod_i\alpha_i$ is already generated  by $\hat{\alpha_1}$, the generators for
$i\geq 3$  can be replace by $m\hat{\alpha_i}$ (which can be realized). 

\vspace{2mm}

Though in case (ii) the order of $[E^*_0]$ is $2$, hence   one has 
more congruences, the proof goes similarly. Therefore we omit it and invite the interested reader to
check the details.
\end{proof}

\end{document}